\documentclass[12pt]{amsart}
\usepackage{verbatim}
\usepackage{amssymb, amsmath}
\usepackage{graphicx}
\usepackage{caption}
\usepackage{subcaption}

\usepackage{xcolor}
\usepackage{color}

\begin{document}
\newtheorem{thm}{Theorem}[section]
\newtheorem*{thm*}{Theorem}%[section]
\newtheorem{lem}[thm]{Lemma}
\newtheorem{prop}[thm]{Proposition}
\newtheorem{cor}[thm]{Corollary}
\newtheorem*{conj}{Conjecture}%[section]
\newtheorem{proj}[thm]{Project}%[section]
\newtheorem{question}[thm]{Question}
\newtheorem{rem}{Remark}[section]

\theoremstyle{definition}
\newtheorem*{defn}{Definition}
\newtheorem*{remark}{Remark}
\newtheorem{exercise}{Exercise}
\newtheorem*{exercise*}{Exercise}

\numberwithin{equation}{section}

\newcommand{\rad}{\operatorname{rad}}

\newcommand{\Z}{{\mathbb Z}} %cph changed from \mathbf
\newcommand{\Q}{{\mathbb Q}}
\newcommand{\R}{{\mathbb R}}
\newcommand{\C}{{\mathbb C}}
\newcommand{\N}{{\mathbb N}}
\newcommand{\FF}{{\mathbb F}}
\newcommand{\fq}{\mathbb{F}_q}
\newcommand{\rmk}[1]{\footnote{{\bf Comment:} #1}}

\renewcommand{\mod}{\;\operatorname{mod}}
\newcommand{\ord}{\operatorname{ord}}
\newcommand{\TT}{\mathbb{T}}
\renewcommand{\i}{{\mathrm{i}}}
\renewcommand{\d}{{\mathrm{d}}}
\renewcommand{\^}{\widehat}
\newcommand{\HH}{\mathbb H}
\newcommand{\Vol}{\operatorname{vol}}
\newcommand{\area}{\operatorname{area}}
\newcommand{\tr}{\operatorname{tr}}
\newcommand{\norm}{\mathcal N} % norm =(\frac{ n+\sqrt{n^2-4}} 2)^2
\newcommand{\intinf}{\int_{-\infty}^\infty}
\newcommand{\ave}[1]{\left\langle#1\right\rangle} %  average
\newcommand{\Var}{\operatorname{Var}}
\newcommand{\Prob}{\operatorname{Prob}}
\newcommand{\sym}{\operatorname{Sym}}
\newcommand{\disc}{\operatorname{disc}}
\newcommand{\CA}{{\mathcal C}_A}
\newcommand{\cond}{\operatorname{cond}} % conductor
\newcommand{\lcm}{\operatorname{lcm}}
\newcommand{\Kl}{\operatorname{Kl}} %Kloosterman sum
\newcommand{\leg}[2]{\left( \frac{#1}{#2} \right)}  % Legendre symbol
\newcommand{\Li}{\operatorname{Li}}

\newcommand{\sumstar}{\sideset \and^{*} \to \sum}

\newcommand{\LL}{\mathcal L} %L-function of u
\newcommand{\sumf}{\sum^\flat}
\newcommand{\Hgev}{\mathcal H_{2g+2,q}}
\newcommand{\USp}{\operatorname{USp}}
\newcommand{\conv}{*}
\newcommand{\dist} {\operatorname{dist}}
\newcommand{\CF}{c_0} % Fejer constant
\newcommand{\kerp}{\mathcal K}

\newcommand{\Cov}{\operatorname{cov}}
\newcommand{\Sym}{\operatorname{Sym}}

\newcommand{\Ht}{\operatorname{Ht}}

\newcommand{\E}{\operatorname{\mathbb E}} % expectation
\newcommand{\sign}{\operatorname{sign}} %sign
\newcommand{\meas}{\operatorname{meas}} %measure
\newcommand{\length}{\operatorname{length}} %length

\newcommand{\divid}{d} % the divisor function

\newcommand{\GL}{\operatorname{GL}}
\newcommand{\SL}{\operatorname{SL}}
\newcommand{\re}{\operatorname{Re}}
\newcommand{\im}{\operatorname{Im}}
\newcommand{\res}{\operatorname{Res}}
 \newcommand{\eigen}{\Lambda} %eigenvalue on rectangle
\newcommand{\tens}{\mathbf t} %Tension
\newcommand{\diam}{\operatorname{diam}}
\newcommand{\fixme}[1]{\footnote{Fixme: #1}}
\newcommand{\new}[1]{{ #1}}

\title[On the Robin spectrum for the equilateral triangle]{ On the Robin spectrum for the equilateral triangle}
\author{Ze\'ev Rudnick and Igor Wigman}

\address{School of Mathematical Sciences, Tel Aviv University, Tel Aviv 69978, Israel} \email{rudnick@tauex.tau.ac.il}
\address{Department of Mathematics, King's College London, UK}\email{igor.wigman@kcl.ac.uk}

\begin{abstract}The equilateral triangle is one of the few planar domains where the Dirichlet and Neumann eigenvalue problems  were explicitly determined, by Lam\'e  in 1833, despite not admitting separation of variables. In this paper, we study the Robin spectrum of the equilateral triangle, which was determined by McCartin in 2004 in terms of a system of transcendental coupled secular equations.

We give uniform upper bounds for the Robin-Neumann gaps, showing that they are bounded by their limiting mean value, which is hence an almost sure bound. The spectrum admits a systematic double multiplicity, and after removing it we study the gaps in the resulting desymmetrized spectrum.
We show  a spectral gap property, that there are arbitrarily large gaps, and also arbitrarily small ones, moreover that the nearest neighbour spacing distribution of the desymmetrized spectrum is a delta function at the origin. We show that for sufficiently small Robin parameter, the desymmetrized spectrum is simple.
\end{abstract}

\thanks{This research was supported by the European Research Council (ERC) under the European Union's Horizon 2020 research and innovation programme (Grant agreement No. 786758) and by the Israel Science Foundation (grant No. 1881/20). }

\dedicatory{Dedicated to Michael Berry for his 80'th birthday}

\date{\today}
\maketitle

\section{Introduction}
The equilateral triangle is one of the few planar domains where the Dirichlet and Neumann eigenvalue problems are explicitly solved, despite not admitting separation of variables. The solution was found by Lam\'e in 1833 \cite{Lame1833},   who also investigated
the Robin\footnote{The term ``Robin boundary condition'' came much later, see \cite{AG} for a historical discussion.} boundary value problem: Denoting by $T$ an equilateral triangle and by $\partial T$ its boundary as in Figure~\ref{fig:equilateral2},  the Robin problem is to solve
 %in the study of the cooling of a right prism with  a base consisting of an equilateral triangle:
% the  Robin eigenvalue problem for an equilateral triangle $T$:
\[
\Delta f+ \lambda f=0\quad { \rm on}\; T, \quad \frac{\partial f}{\partial n} +\sigma f=0\;  {\rm on}\; \partial T
\]
where $\frac{\partial}{\partial n}$ is the derivative in the outward pointing normal direction, and $\sigma>0$ is the Robin parameter (which we take to be constant).

%\begin{figure}[ht]
%\begin{center}
%\includegraphics[height=40mm]{equilateralfig.png}
%\caption{Equilateral triangle with inscribed radius $r$.}
%\label{fig:equilateral}
%\end{center}
%\end{figure}

 \begin{figure}[ht]
\begin{center}
\includegraphics[height=60mm]{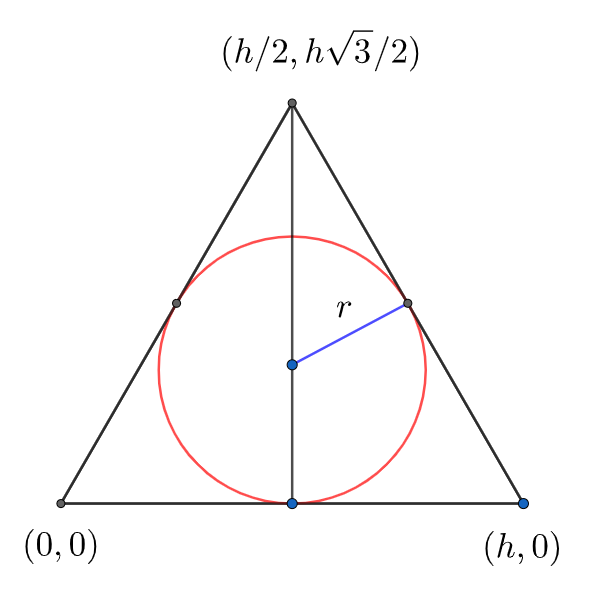}
\caption{An equilateral triangle  of side length $h$. The inscribed circle has radius $r=h/(2\sqrt{3})$.}
\label{fig:equilateral2}
\end{center}
\end{figure}

%According to McCartin \cite{McCartin2004}
Lam\'e only determined the Robin eigenfunctions possessing $120^{\circ}$  rotational symmetry, and it is only in 2004 that
   McCartin \cite{McCartin2004, McCartinbook} completely determined the eigenproblem, showing that all of the eigenfunctions are trigonometric polynomials\footnote{The only polygonal domains where all Dirichlet or Neumann eigenfunctions are trigonometric are rectangles, and the equilateral, hemi-equilateral and right isosceles triangles \cite{McCartin2008}, see \cite{RBNTV} for a higher dimensional version}, and that the eigenvalues are determined by a system of transcendental coupled secular equations as follows:
    Define auxiliary parameters
%\begin{equation*}%\label{range of L,M,N}
$L\in (-\pi/2,0]$,  $M,N\in [0,\pi/2)$,
%\end{equation*}
which are required to satisfy the coupled system of equations
\begin{equation}\label{coupled system intro}
\begin{split}
\Big(2L-M-N-(m+n)\pi \Big)\tan L &= 3r \sigma
\\
\Big(2M-N-L+m\pi \Big)\tan M &= 3r \sigma
\\
\Big(2N-L-M+n\pi \Big) \tan N &= 3r \sigma .
\end{split}
\end{equation}
The corresponding  Robin eigenvalues are % parameterized by pairs of integers $m,n\geq 0$, in the form
\begin{equation}\label{def of eigen}
\eigen_{m,n} (\sigma)%= \frac{2\pi^2}{27 r^2}(\mu^2+\nu^2+\lambda^2)
=\frac{4\pi^2}{27 r^2}(\mu^2+\nu^2+\mu\nu)
\end{equation}
where
\[
 %\lambda=\frac{2L-M-N}{\pi}-m-n, \;
 \mu=\frac{2M-N-L}{\pi}+m,\quad  \nu = \frac{2N-L-M}{\pi}+n .
\]
Note that there is a systematic multiplicity of order $2$ coming from the symmetry $\eigen_{m,n} = \eigen_{n,m}$,
\new{and we will refer to $\{\eigen_{m,n} (\sigma)\}_{0\le m\le n} $ as the {\em desymmetrized} Robin spectrum}.

For $\sigma\geq 0$, let $\lambda_{n}^{\sigma}$ denote the $n$-th eigenvalue of the Robin Laplacian on the equilateral triangle,
arranged by size and repeated with multiplicities (the case $\sigma=0$ are the Neumann eigenvalues). We will study a number of aspects of the Robin spectrum of the equilateral triangle.

%\marginpar{insert ref to RW rectangle~\cite{RW rectangle} }
%\marginpar{Do we want to add reference to Sieber-\textcolor{red}{Smilansky} and Berry-Dennis?}

In the first part of the paper, we study the Robin-Neumann gaps
\[
d_n(\sigma):=\lambda_{n}^{\sigma}-\lambda_{n}^{0}  ,
\]
see  Figure~\ref{equilateral gaps}.
As is the case for any bounded piecewise smooth planar domain, the RN gaps have a
limiting mean value \cite{RWY}, which  equals
\[
\bar d :=\lim_{N\to \infty} \frac 1N\sum_{n=1}^N d_n(\sigma) = \frac{2\length \partial T}{\area T} \sigma = \frac 4r \sigma ,
\]
where $r$ is the radius of the inscribed circle.
%We will show  a uniform sharp upper bound on the RN gaps for the equilateral triangle:
 Remarkably, for the equilateral triangle, the limiting mean value is also an upper bound:

\begin{thm} \label{thm: diffs for triangle} % avoid duplication with statement in section
We have  $d_n(\sigma)<\bar d$  for all $n$.
\end{thm}
 %Since the limiting mean value $\bar d$ is also an upper bound it is optimal.
% \textcolor{red}{In Appendix~\ref{app:rectangle} we will show that this also holds for rectangles, improving \cite{RW rectangle}. }
% \marginpar{do we want the rectangle?}
For most domains, we do not expect a uniform upper bound, e.g. we expect that for the disk there are arbitrarily large RN gaps,
but cannot prove this for any example of a planar domain, see \cite{RWhemisphere} for the hemisphere.

As a consequence of Theorem~\ref{thm: diffs for triangle}, we show:
\begin{cor}\label{cor: Ae convergence of gaps}
\new{The RN gaps tend to the mean value along a density one sequence of eigenvalues.}
\end{cor}
  %A different task is to control the number of exceptions.

%\begin{question}
%\textcolor{red}{ Prove uniform lower bounds on RN gaps; we know lower bounds for any domains with smooth boundary (which are star-shaped). We don't know of any domain where there is no uniform lower bound.}
%\end{question}

%For both we have uniform upper bounds for the gaps.
 \begin{figure}[ht]
       \includegraphics[height=60mm]{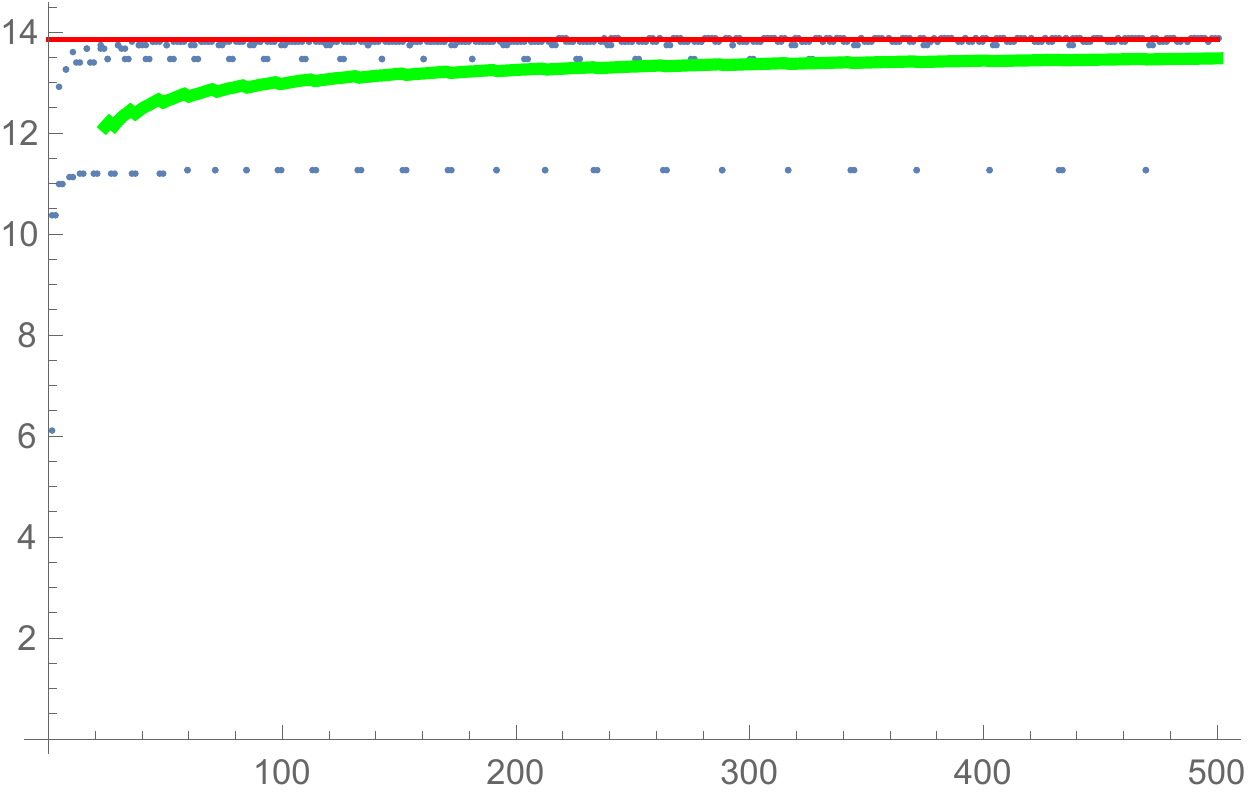}
    \caption{The first 500 RN gaps  for the equilateral triangle  %\eqref{fig:equilateral}
    with side length $1$, with $\sigma=1$. The solid (red) line is the limiting mean value $2\length(\partial T)/\area(T)=8\sqrt{3}=13.8564\dots$. Note that all the gaps are below the limiting mean value, as is proved in Theorem~\ref{thm: diffs for triangle}.
   }
   \label{equilateral gaps}
\end{figure}

We use the results on the  RN gaps to deduce information on the asymptotics of the Robin spectrum of the equilateral triangle by comparing it to the Neumann spectrum:

\begin{cor}
\label{cor:arb large gaps}
\new{For fixed $\sigma>0$, there are arbitrarily large gaps in the Robin spectrum $\{\lambda_n^\sigma\}$.}
\end{cor}
  This is sometimes called the ``spectral gap property'' and is useful in a variety of applications. An example is to show the existence of inertial manifolds in dissipative reaction-diffusion equations \cite{FST, MPS}
$$\frac{\partial u}{\partial t}=\nu \Delta u +g(u),$$
for $u$ on a domain satisfying suitable boundary conditions, with $g$ a suitable nonlinear function, and where $\nu>0$ is a parameter, see \cite{Kwean} for the case of the equilateral triangle. There are very few instances of planar domains where the existence of arbitrarily large gaps in the spectrum (with any boundary condition) is known. The question is open even  for the Dirichlet spectrum of the rectangle having the golden mean as its aspect ratio.

We can also show that there are arbitrarily small nonzero gaps in the spectrum.
In fact, we have a stronger result:

\begin{thm}
\label{thm:nearest neigh delta}
The distribution of nearest neighbour gaps in the desymmetrized spectrum is a delta function at the origin, i.e. for any fixed $x>0$,
\[
\lim_{N\to \infty}\frac 1N \#\{ n\leq N: \lambda_{n+1}^\sigma-\lambda_n^\sigma \leq x \} = 1
\]
 \end{thm}
We emphasize that in this paper, $\sigma$ is fixed; it is of great interest to study the spacings when $\sigma$ grows with the eigenvalue, see the discussion by  Sieber,  Primack,  Smilansky,  Ussishkin and   Schanz
 \cite{SS},  and by Berry and Dennis \cite{BD}.

In the second part of the paper  we examine spectral multiplicities (or ``modal degeneracies'').  For the Dirichlet or Neumann spectrum of the equilateral triangle, there are large multiplicities of arithmetic origin; the same holds for the hemi-equilateral (half of an equilateral triangle) and right isosceles  triangles, but there are other triangles with accidental degeneracies, see the paper of  Berry and Wilkinson  \cite{BerryWilkinson} for an exploration of these ``diabolical points''. Hillairet and Judge \cite{HJ} showed that for  {\em almost all}\footnote{In the sense of Lebesgue measure on the space of triangles of fixed area, which can be parameterized by   triples of angles which sum to $\pi$.} triangles the Dirichlet spectrum is simple, but their method does not give a single explicit example.
For the Robin spectrum on the equilateral triangle, there is a systematic doubling due to the symmetry $(m,n)\mapsto (n,m)$ in \eqref{def of eigen}.
McCartin  \cite[\S 8]{McCartin2004} observed that there are additional degeneracies for $\sigma\gg 1$.  We will show that for small $\sigma>0$, there are no other degeneracies:
\begin{thm}\label{thm:simplicity}
There is some $\sigma_0>0$ so that there are no multiplicities in the Robin spectrum for $0<\sigma<\sigma_0$ except for the systematic doubling.
\end{thm}
A similar result holds for the square; however, for rectangles whose squared aspect ratio is irrational, there are multiplicities for arbitrarily small $\sigma>0 $ \cite{RW rectangle}, showing the special arithmetic nature of the result.

For the proof of Theorem~\ref{thm:simplicity}, we partition the spectrum into clusters
$$\mathcal C_R(\sigma) = \{\Lambda_{m,n}(\sigma): m,n\geq 0, m^2+mn + n^2=R^2\},
$$
consisting of all Robin eigenvalues given by \eqref{def of eigen}, that, for $\sigma=0$, correspond
to the common Neumann eigenvalue satisfying $$\frac{4\pi^2}{27r^2}(m^2+mn + n^2)= \Lambda_{m,n}(0)=\frac{4\pi^2}{27r^2} R^2$$
with some $m,n\ge 0$ integers, for the given $R>0$. At $\sigma=0$, these clusters are well separated, as the Neumann eigenvalues are
multiples by $\frac{4\pi^2}{27r^2}$ of integers. As $\sigma$ varies, different clusters remain separated for small $\sigma$ due to our upper bound on the Robin-Neumann gaps (Theorem~\ref{thm: diffs for triangle}). This reduces the problem to showing that there is some $\sigma_0>0$ for which all of the clusters break up completely (except for a systematic double multiplicity) for all $0<\sigma<\sigma_0$. To prove this requires a detailed study of the secular equations \eqref{coupled system intro} governing the eigenvalues, which takes up sections \ref{sec: pfs of lemmas 1 and 2}, \ref{sec: pfs of lemmas 3 and 4} and \ref{sec: pf of lemma5}.

\section{Background on the equilateral triangle}
%\subsection{Coordinate systems}
We consider an equilateral triangle $T$ of side length $h$. Denote by
$$r= \frac{h}{2\sqrt{3}} $$
the radius of the inscribed circle. The area of $T$  is then
\[
\area(T)=   \frac{\sqrt{3}h^2}{4} = 3\sqrt{3}r^2.
\]
We use  Cartesian coordinates $(x,y)$ so that the vertices are located at $\{(0,0), (0,h),(h/2, h\sqrt{3}/2)\}$
(Figure~\ref{fig:equilateral2}).

\subsection{Neumann eigenfunctions}
%The Neumann eigenfunctions satisfy
%\[
%\Delta f+ \lambda f=0, \quad \frac{\partial f}{\partial n}  =0
%\]
%where $\frac{\partial}{\partial n}$ is the derivative in the outward pointing normal direction.

The eigenfunctions are either symmetric or antisymmetric w.r.t the altitude of the triangle, that is the line $x=h/2$.
A complete set of orthogonal Neumann eigenfunctions is
\[
\begin{split}
T_{m,n}^{s/a}(x,y) &= \cos\left(\frac{\pi \ell}{3r}\left(3r-y \right)  \right)\left\{\begin{matrix} \cos \\ \sin \end{matrix}\right\}  \left(\frac{\sqrt{3}\pi(m-n)}{9r}\left(x-\sqrt{3}r\right)\right)
\\& +
\cos\left(\frac{\pi m}{3r}\left(3r-y\right)\right)  \left\{\begin{matrix} \cos \\ \sin \end{matrix}\right\} \left(\frac{\sqrt{3}\pi \left( n-\ell \right)}{9r}\left(x-\sqrt{3}r\right)\right)
\\&
+
\cos\left(\frac{\pi n}{3r}\left(3r-y \right) \right) \left\{\begin{matrix} \cos \\ \sin \end{matrix}\right\} \left(\frac{\sqrt{3}\pi \left(\ell-m \right)}{9r}\left(x-\sqrt{3}r \right) \right)
\end{split}
\]
where for the symmetric eigenfunctions $T_{m,n}^s$ we take $0\leq m\leq n$ and cosine, and for the antisymmetric ones $T_{m,n}^a$ we take $0\leq m<n$ and sine.
Here the $m,n \geq 0$ are integers, and $m,n,\ell$ satisfy
\[
m+n+\ell=0
\]
 with the corresponding eigenvalue being
\[
\Lambda_{m,n}(0):=\frac{2\pi^2}{27 r^2}\left(m^2+n^2+\ell^2\right)=\frac{4\pi^2}{27 r^2}\left( m^2+mn+n^2 \right) .
\]

There are high multiplicities in the Neumann spectrum of the equilateral triangle, coming from the fact that for integers  which can be written in  the form $m^2+mn+n^2$  there are ``typically''  many ways to do so. This is a well-understood number theoretic issue, completely similar to the problem of representation as a sum of two squares.
The squared $L^2$ norm of the eigenfunctions is \cite[\S 8.1]{McCartin Neumann}
\[
||T_{m,n}^{a/s}||_2^2=\int_T (T_{m,n}^{s/a})^2 = \frac{9\sqrt{3}r^2}{4} %=\frac{3\sqrt{3}}{16}
, \quad m<n
\]
and
\[
||T_{m,m}^{s}||_2^2 = \frac{9\sqrt{3}r^2}{2}%=\frac{3\sqrt{3}}{8}
, \quad m>0 .
\]
%We define normalized Neumann eigenfunctions by
%\begin{equation}\label{Nemann efs}
%u_{m,n}^{s/a} = \frac{ T_{m,n}^{s/a}}{ ||T_{m,n}^{a/s}||_2 }
%\end{equation}

\subsection{Robin eigenfunctions}

The eigenfunctions are either symmetric or antisymmetric w.r.t the altitude of the triangle, that is the line \new{$x=h/2$}. McCartin showed that a complete set of orthogonal eigenfunctions is
\[
\begin{split}
T_{m,n}^{s/a}(x,y) &= \cos\left(\frac{\pi \lambda}{3r}\left(3r-y\right)-\delta_1\right)\{\begin{matrix} \cos \\ \sin \end{matrix}\}  \left(\frac{\sqrt{3}\pi \left(\mu-\nu \right)}{9r} \left(x-\sqrt{3}r \right)\right)
\\& +
\cos\left(\frac{\pi \mu}{3r}\left(3r-y \right)-\delta_2 \right)  \{\begin{matrix} \cos \\ \sin \end{matrix}\} \left(\frac{\sqrt{3}\pi \left(\nu-\lambda \right)}{9r} \left( x-\sqrt{3}r \right) \right)
\\&
+
\cos \left( \frac{\pi \nu}{3r} \left(3r-y \right)-\delta_3 \right) \{\begin{matrix} \cos \\ \sin \end{matrix}\} \left(\frac{\sqrt{3}\pi \left( \lambda-\mu \right)}{9r} \left(x-\sqrt{3}r \right) \right)
\end{split}
\]
\new{with some $\delta_{1},\delta_{2},\delta_{3}\in\R$,} where for the symmetric eigenfunctions $T_{m,n}^s$ we take $0\leq m\leq n$ and cosine, and for the antisymmetric ones $T_{m,n}^a$ we take $0\leq m<n$ and sine.
Here $\mu,\nu,\lambda$ (depending on $m$, $n$ and the Robin constant $\sigma$) are chosen subject to
\[
\mu+\nu+\lambda=0
\]
and $\mu,\nu \geq 0$ are determined by a set of transcendental equations (imposed by requiring that the corresponding eigenfunctions  satisfy the Robin condition on the boundary):  Define auxiliary parameters
\begin{equation}\label{range of L,M,N}
L\in (-\pi/2,0], \quad M,N\in [0,\pi/2)
\end{equation}
 and set
\[
 \lambda=\frac{2L-M-N}{\pi}-m-n, \;  \mu=\frac{2M-N-L}{\pi}+m,\;  \nu = \frac{2N-L-M}{\pi}+n .
\]
Then $L,M,N$ are required to satisfy the coupled system of equations
\begin{equation}\label{coupled system}
\begin{split}
\Big(2L-M-N-(m+n)\pi \Big)\tan L &= 3r \sigma
\\
\Big(2M-N-L+m\pi \Big)\tan M &= 3r \sigma
\\
\Big(2N-L-M+n\pi \Big) \tan N &= 3r \sigma ,
\end{split}
\end{equation}
 see \cite{McCartin2004} for existence and uniqueness of solutions. 
 
 The corresponding eigenvalues are % parameterized by pairs of integers $m,n\geq 0$, in the form
\begin{equation}\label{expression for evs}
\eigen_{m,n} (\sigma)= \frac{2\pi^2}{27 r^2}(\mu^2+\nu^2+\lambda^2)
=\frac{4\pi^2}{27 r^2}(\mu^2+\nu^2+\mu\nu) .
\end{equation}
\new{One may find some examples of plots of $\Lambda_{m,n}(\cdot)$ in Figure \ref{fig:eigenval plots}.}
Note that there is a systematic multiplicity of order $2$ coming from the symmetry $\eigen_{m,n} = \eigen_{n,m}$.
We refer to \cite{Finjordetal} for a computation of the $L^2$ norm of the eigenfunctions.

\begin{figure}[ht]
\begin{center}
\includegraphics[height=60mm]{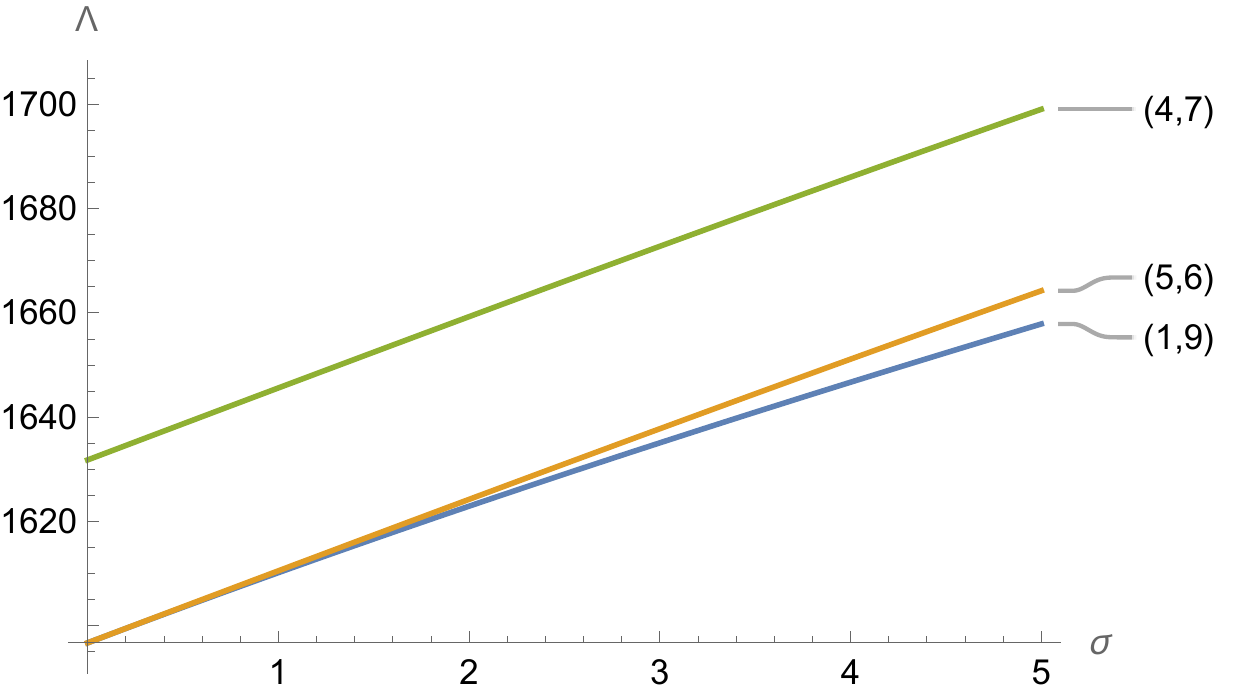}
\caption{Plots of $\Lambda_{m,n}(\sigma)$ with $(m,n)=(1,9),(5,6),(4,7)$. Note that $\Lambda_{1,9}(0)=\Lambda_{5,6}(0)$.}
\label{fig:eigenval plots}
\end{center}
\end{figure}

\section{A uniform upper bound for the RN gaps: Proof of Theorem~\ref{thm: diffs for triangle} }\label{sec: upper bound}

% For $\sigma\geq 0$, let $\lambda_{n}^{\sigma}$ denote the $n$-th eigenvalue of the Robin Laplacian on the equilateral triangle, arranged by size and repeated with multiplicities (the case $\sigma=0$ are the Neumann eigenvalues).
Our goal is to show that for the equilateral triangle, the Robin-Neumann gaps are bounded above
 by their limiting mean value, which we recall equals
\[
\bar d:=\lim_{N\to \infty} \frac 1N\sum_{n=1}^N d_n(\sigma) = \frac{2\length \partial T}{\area T} \sigma = \frac 4r \sigma.
\]
%\begin{thm}\label{thm: diffs for triangle}
%For the equilateral triangle, we have an upper bound for the RN gaps:
%\[
%d_n(\sigma) \leq \frac 4r \sigma=\bar d.
%\]
%\end{thm}

\begin{proof}
We will show that
\begin{equation}\label{deviation of eigen_{m,n}}
0<\eigen_{m,n}(\sigma)-\eigen_{m,n}(0) <   \overline{d}=\frac{4}{   r} \sigma .
\end{equation}
Given that, to pass from the $\eigen_{m,n}(\sigma)$ to its analogue for the {\em ordered} eigenvalues is the same argument as for the rectangle (see \cite[\S 8.2]{RWY}) that we reproduce here for the sake of completeness, albeit briefly. Namely, recall that $\{\lambda_{n}^{\sigma}\}_{n\ge 0}$ is the Robin spectrum corresponding to the Robin parameter $\sigma$ (in non-decreasing order), and, given $k\ge 1$, consider the closed interval $$I_{k}:=[0,\lambda_{k}^{0}+\overline{d}]\subseteq \R.$$ Then, thanks to the inequality \eqref{deviation of eigen_{m,n}} to be proved immediately below, $I_{k}$ is bound to contain all of $\lambda_{n}^{\sigma}< \lambda_{n}^{0}+\overline{d},$ for $n\le k$, i.e. $I_{k}$ contains at least $(k+1)$ of the eigenvalues $\{\lambda_{n}^{\sigma}\}$, implying, in particular, that $$\lambda_{k}^{\sigma}\le \lambda_{k}^{0}+\overline{d},$$ sufficient to deduce the claimed analogue of \eqref{deviation of eigen_{m,n}} for the ordered Robin eigenvalues.

We now turn to proving \eqref{deviation of eigen_{m,n}}. To this end we rewrite the equations \eqref{coupled system}
in a compact form as follows: Set
\[
\begin{split}
 m_1=m, \quad   & m_2=n, \quad  m_3=-(m+n),
 \\
\mu_1=\mu, \quad  &\mu_2=\nu,\quad  \mu_3=\lambda,
\\
 M_1=M, \quad &M_2=N, \quad M_3=L
\end{split}\]
so that
\begin{equation}\label{sum of mu is 0}
\mu_1+\mu_2 +\mu_3=0=m_1+m_2+m_3
\end{equation}
and
\begin{equation}\label{mu in terms of M}
\mu_j = m_j + \frac 1\pi(2M_j-M_i-M_k)
\end{equation}
where $\{i,j,k\}=\{1,2,3\}$,
and the system \eqref{coupled system} becomes
 \begin{equation*}
\   \mu_j \tan M_j =\frac{3r\sigma}{\pi}, \quad j=1,2,3  .
\end{equation*}
  Therefore, since $|M_j|<\pi/2$,
 \begin{equation}\label{M_J small}
  |\mu_j M_j| < |\mu_j\tan M_j| =\frac{3r}{\pi} \sigma .
 \end{equation}

 Now consider the difference (compare \eqref{expression for evs})
 \[
 \eigen_{m,n}(\sigma)-\eigen_{m,n}(0)=\frac{2\pi^2}{27 r^2}\sum_{j=1}^3(\mu_j^2-m_j^2) .
 \]
 % From \eqref{M_J small}, uniformly in $m,n$ as $\sigma\to 0$  \marginpar{Check $m=0$ or $n=0$}
 %\[
 %\mu_j = m_j +  \frac 1\pi (2M_j-M_i-M_k) = m_j+O(\sigma)
 %\]
We have
 \[
 \begin{split}
 \mu_j^2-m_j^2  &= (\mu_j-m_j)(2\mu_j-(\mu_j-m_j))
 \\
 &=   2\mu_j(\mu_j-m_j)-(\mu_j-m_j)^2 
 \\&\new{\leq 2\mu_j(\mu_j-m_j)}
   \\
 &= \frac{2}{\pi} \mu_j (2M_j-M_i-M_k)  .
 \end{split}
 \]
 on inserting \eqref{mu in terms of M}. Therefore %, uniformly in $m,n$,
 \[
 \begin{split}
0< \eigen_{m,n}(\sigma)-\eigen_{m,n}(0) &= \frac{2\pi^2}{27 r^2} \sum_{j=1}^3 (\mu_j^2-m_j^2)
\\
&\leq  \frac{4\pi}{27 r^2} \sum_{j=1}^3 \mu_j (2M_j-M_i-M_k)   .
 \end{split}
 \]
 Recalling that  $\{i,j,k\} = \{1,2,3\}$ and using \eqref{sum of mu is 0} gives
 \[
 \sum_{j=1}^3 \mu_j (M_i+M_k) = \sum_j M_j(\mu_i+\mu_k) = -\sum_j M_j\mu_j
 \]
 and so
 \[
 \sum_{j=1}^3 \mu_j (2M_j-M_i-M_k)  = 2 \sum_{j=1}^3 \mu_j  M_j - \sum_{j=1}^3 \mu_j (M_i+M_k)  = 3\sum_{j=1}^3M_j\mu_j .
 \]
 Inserting \eqref{M_J small} gives that this is $\leq 27\frac{r}{\pi}  \sigma$ and hence
 \[
 0<\eigen_{m,n}(\sigma)-\eigen_{m,n}(0) \leq   \frac{4\pi}{27 r^2}  \cdot   \frac{27r}{\pi}  \sigma  = \frac{4}{   r} \sigma
 \]
  proving \eqref{deviation of eigen_{m,n}}.
  \end{proof}

\section{Almost sure convergence of the RN gaps: Proof of Corollary~\ref{cor: Ae convergence of gaps}}
A tautological consequence of Theorem~\ref{thm: diffs for triangle}, which says that all RN gaps (which are positive) are bounded by their limiting mean value, is that \new{{\em almost all}} RN gaps converge to the limiting mean value $\bar d$.
%$$\bar d = \lim \frac 1N \sum_{n=1}^N d_n(\sigma) = \frac{2\length \partial T}{\area T} \sigma .$$
 % \begin{prop}
 % Fix $\sigma>0$. Then for almost all $n$, we have
 % \[
 % d_n(\sigma) = \bar d +o(1).
 %\]
%  \end{prop}
\begin{proof}
  Indeed, let $d_n\geq 0$ be a sequence of non-negative numbers, which has a limiting mean value
  $$\bar d:=\lim_{N\to \infty}\frac 1N \sum_{\new{n}=1}^N d_n,$$
  and assume that for all $n$ we have $d_n\leq \bar d$. Then we claim that necessarily, for almost all $n$, we have $d_n = \bar d +o(1)$ as $n\to \infty$. On the contrary, assume that 
  there is some $\delta>0$ so that the set
  \[
  \mathcal N_\delta:=\{n: d_n\leq \bar d-\delta\}
  \]
  satisfies:
  \begin{equation*}%\label{limsup set}
  \limsup \frac 1N\#\mathcal N_\delta\cap [1,N] =c>0.
  \end{equation*}
Thus we are guaranteed an infinite sequence $\mathcal S$ of $N$'s satisfying
\[
\#\mathcal N_\delta\cap [1,N] >cN/2.
\]

  For all $N\geq 1$, we can compute the mean value as
  \[
  \frac 1N \sum_{n=1}^N d_n =  \frac 1N \sum_{\substack{n\leq N\\ n\in \mathcal N_\delta}}  d_n +
   \frac 1N \sum_{\substack{n\leq N\\ n\notin \mathcal N_\delta}}  d_n
   \leq
    \frac 1N \sum_{\substack{n\leq N\\ n\in \mathcal N_\delta}} ( \bar d-\delta) +  \frac 1N \sum_{\substack{n\leq N\\ n\notin \mathcal N_\delta}} \bar d
  \]
  where we have used $d_n\leq \bar d$ for $n\notin \mathcal N_\delta$. In particular, for $N\in \mathcal S$,
    \begin{multline*}
  \frac 1N \sum_{n=1}^N d_n  \leq ( \bar d-\delta )  \frac 1N\#\mathcal N_\delta\cap [1,N] +\bar d  \frac 1N\#\{n\notin \mathcal N_\delta, n\leq N\}
\\  =\bar d-\delta \frac 1N\#\mathcal N_\delta\cap [1,N]  \leq \bar d -\delta \frac c2
  \end{multline*}
  and so
  \[
 \bar d =  \lim_{\substack{N\to \infty\\ N\in \mathcal S
}}  \frac 1N \sum_{n=1}^N d_n \leq \bar d -\delta \frac c2 <\bar d
  \]
  which is a contradiction.
  \end{proof}

\section{Large gaps in the Robin spectrum: \new{Proof of Corollary \ref{cor:arb large gaps}}}

%HERE erased

\begin{proof}
  Since the Robin spectrum clusters at a bounded distance around the Neumann spectrum
  $\{\frac{4\pi^2}{27 r^2}( m^2+mn+n^2 ) :m,n\geq 0\}$, it suffices to observe that the Neumann spectrum has arbitrarily large gaps. This well known arithmetic fact admits a quick proof by noting that for  integers of the form $m^2+mn +n^2$, the prime decomposition can only contain primes of the form $p=3k+2$ to an even power (see e.g. \cite[Chapter 9.1]{IR}). Let $p_1=2, p_2=5,\dots, p_K$ be the first $K$ primes congruent to $2 \bmod 3$. Using the
  Chinese Remainder Theorem we find $n$ satisfying $n=-j+p_j \bmod p_j^2$ for $j=1,\dots,K$. Then $n+1,n+2,\dots ,n+K$ are not of the form $x^2+xy+y^2$ because $n+j=0\bmod p_j$ while $n+j\neq 0\bmod p_j^2$. Thus we found a gap of size $\geq \frac{4\pi^2}{27r^2}\cdot  K$ in the Neumann spectrum.
  \end{proof}
%the sieve of Erathosthenes to deduce that such integers have zero density, hence have to have unbounded gaps.

  We can extract qualitative results from the finer results known about gaps between values of binary quadratic forms:
  %The number of such integers up to $x$ is asymptotically $cx/\sqrt{\log x}$ for a suitable $c>0$, hence the average gaps between such integers near $x$ grows like $\sqrt{\log x}/c$.
In 1982,  Richards \cite{Richards} proved that the maximal gap $g(x)$ among  integers of the form $m^2+mn+n^2$ up to $x$ is at least $ (\frac 13 -o(1)) \log x$ as $x\to \infty$, see \cite{DE, KK} for improvement to the constant. Hence, if we denote by
\[
g_\sigma(x)=  \max \left(\lambda_{n+1}^\sigma-\lambda_n^\sigma: \lambda_n^\sigma\leq x\right),
\]
then, with the help of Theorem~\ref{thm: diffs for triangle}, we obtain the bound
\[
g_\sigma(x) \gg \log x.
\]

\section{Spacings: \new{Proof of Theorem \ref{thm:nearest neigh delta}}}

Fix $\sigma>0$ and denote by $\lambda_n^\sigma$ the Robin spectrum ($\sigma=0$ being the Neumann spectrum) and let
$$x_\sigma(n) = c(\lambda_{n+1}^\sigma-\lambda_n^\sigma) , \quad c=\frac{4\pi}{\area T}$$
be the normalized nearest neighbour gaps, whose mean value is unity by Weyl's law. Let
$$\tilde P_\sigma(t,N):=\frac 1N \#\{ n\leq N: x_\sigma(n)<t\}
%made it a strict inequality
%\frac 1N \#\{ n\leq N: x_\sigma(n)\leq t\}
$$
be the cumulative distribution function of the $x_\sigma(n)$.

\new{Note that if for a pair of tuples $(m,n)$ and $(m',n')$ representing consecutive Neumann energies one has $$m^2+n^2+mn=m'^2+n'^2+m'n',$$ then $\Lambda_{m,n}(0)=\Lambda_{m',n'}(0)$, so that the corresponding nearest neighbour gap vanishes. Hence for the Neumann spectrum, all the nearest gaps vanish except for at most the number of integers representable by the form $m^2+n^2+mn$, whose number of those energies $\le X$ is  $O\left(\frac{X}{\sqrt{\log{X}}}\right)$ by ~\cite{James}. On the other hand, by Weyl's law, the total number of energies $\le X$ is proportional to $X$, hence most of the gaps vanish precisely, implying, in particular, that the limiting spacing distribution is the delta function.
Hence, in the Neumann case, the \new{limiting} spacing distribution is a delta-function at the origin: we have for any $t>0$,}
\begin{equation}\label{cumulative spacing for neumann}
\lim_{N\to \infty} \tilde P_0(t ,N) = 1 .
\end{equation}
%We will show that for $\sigma>0$,
%\[
%\lim_{N\to \infty} \left(\tilde  P_\sigma(x,N)-\tilde P_0(x,N) \right) =0
%\]
%which establishes our claim.

\begin{proof}[\new{Proof of Theorem \ref{thm:nearest neigh delta}}]

\new{By Corollary \ref{cor: Ae convergence of gaps}, the bulk of Robin spectrum is obtained from the Neumann spectrum by an approximately constant shift, therefore the spacing distribution remains unchanged.}
Denote by $d_\sigma(n) = \lambda_n^\sigma-\lambda_n^0$ the Robin-Neumann gaps and $\bar d $ their limiting mean value. Fix $\epsilon>0$ and let $\mathcal S$ be the set of integers $n$ so that
$$\bar d-\epsilon <d_\sigma(n)$$
(and also $d_n(\sigma)<\bar d$).
We showed that $\mathcal S$ has density one (Corollary~\ref{cor: Ae convergence of gaps}).
Therefore, the set
$$\mathcal S_2:=\{n\geq 1: n\in \mathcal S \mbox{  and  } n+1\in \mathcal S\}$$
also has density one.

For $n\in \mathcal S_2$, we compute the difference of the normalized gaps $x_\sigma(n)$ and $x_0(n)$:
\[
\begin{split}
\frac 1c\left( x_\sigma(n) -x_0(n) \right)
&= (\lambda_{n+1}^\sigma-\lambda_n^\sigma) - ( \lambda_{n+1}^0-\lambda_n^0)
\\&
= (  \lambda_{n+1}^\sigma-\lambda_{n+1}^0) - (\lambda_n^\sigma-\lambda_n^0)
%\\&
= d_\sigma(n+1)-d_\sigma(n)  .
\end{split}
\]
Since $d_\sigma(n+1),d_\sigma(n) \in (\bar d-\epsilon, \bar d)$ we obtain $ d_\sigma(n+1)-d_\sigma(n)  \in (-\epsilon,\epsilon)$ so that  for all $n\in \mathcal S_2$,
\[
 x_\sigma(n) -x_0(n)\in (-c\epsilon, c\epsilon).
\]

Fix $t>0$, and take $\epsilon<t/c$. Then for $n\in S_2$, if $x_\sigma(n) < t$ then $x_0(n)<t+c\epsilon$, while $x_0(n)<t-c\epsilon$ implies that $x_\sigma(n)<t$. Thus
\[
\{ n\in \mathcal S_2: x_\sigma(n)<t \}  \subseteq \{n\in \mathcal S_2: x_0(n)<t+c\epsilon\}
\]
and
\[
 \{ n\in \mathcal S_2: x_\sigma(n)<t \}  \supseteq \{n\in \mathcal S_2: x_0(n)<t-c\epsilon\} .
\]
\new{On the other hand,} we have
\begin{multline*}
0\leq \tilde P_\sigma(t,N) - \frac 1N \#\{n\leq N, n\in \mathcal S_2: x_\sigma(n) < t \}
\\
\leq \frac 1N\#\{n\leq N: n\notin \mathcal S_2\} = o(1) ,
\end{multline*}
and likewise for $\sigma=0$, hence
\[
\new{\tilde P_0}(t-c\epsilon,N) +o(1) \leq \tilde P_\sigma(t,N)\leq \tilde P_0(t+c\epsilon,N) +o(1) .
\]
Since $\epsilon>0$ is arbitrary,   for any fixed $t>0$, we obtain  by \eqref{cumulative spacing for neumann}
\[
\lim_{N\to \infty} \tilde P_\sigma(t,N) = 1
\]
  which gives our claim.
 \end{proof}

\section{Simplicity of the desymmetrized spectrum: overview of the proof of Theorem~\ref{thm:simplicity}}
\subsection{Review of notation }
We recall notation:
%\fbox{\textcolor{red}{change and merge later}}
For integers $m,n\geq 0$, and $\sigma\ge 0$, we defined the variables
\begin{equation}
\label{eq:LMN range}
L=L_{m,n}(\sigma)\in \bigg(-\frac \pi 2,0\bigg],\; M=M_{m,n}(\sigma),N=N_{m,n}(\sigma)\in \bigg[0,\frac \pi 2\bigg)
\end{equation}
given by solutions of the system
\begin{equation}
\label{eq:sec eq}
\begin{cases}
\left(2L-M-N-\left(m+n\right)\pi  \right)\tan{L} &=3r\sigma \\
\left(2M-N-L+m\pi  \right)\tan{M} &=3r\sigma \\
\left(2N-L-M+n\pi  \right)\tan{N} &=3r\sigma
\end{cases}.
\end{equation}
The variables $\mu,\nu$ were defined as %\marginpar{Do we need $\lambda$?}
\[
\mu=\frac{2M-N-L}{\pi}+m, \;\nu=\frac{2N-L-M}{\pi}+n,
\]
and, finally, the Robin eigenvalues with parameter $\sigma$ are:
\[
 \Lambda_{m,n}(\sigma):= \frac{4\pi^{2}}{27r^{2}}\left(\mu^{2}+\nu^{2}+\mu\nu\right).
\]

%For every $m,n\geq 0$ and $\sigma \ge 0$ one has   $$\Lambda_{m,n}(\sigma)=\Lambda_{n,m}(\sigma),$$ so that there is a systematic double multiplicity.
%We remove it by assuming that $0\le m\le n$, and call the resulting set of eigenvalues the ``desymmetrized spectrum''.
To prove that the desymmetrized spectrum is simple (Theorem~\ref{thm:simplicity}), it is needed to show that
%\begin{thm}\label{thm:spec nondeg}
there exists $\sigma_{0}>0$ so that for all $\sigma\in (0,\sigma_{0})$, one has $\Lambda_{m,n}(\sigma)\neq \Lambda_{m',n'}(\sigma)$ for all pairs  $(m,n)\neq (m',n')$ with $0\leq m\leq n$, $0\leq m'\leq n'$.
We adopt the notation
\begin{equation}
\label{eq:R radius def}
R^{2}=R^{2}(m,n):=\frac{27 r^2}{4\pi^2}\Lambda_{m,n}(0)=m^{2}+n^{2}+mn,
\end{equation}
and
\[
F_R(m,n) = \frac{R^4}{m^2n^2(m+n)^2}  = \frac 1{m^2}+\frac 1{n^2}+\frac 1{(m+n)^2} .
\]

\subsection{Key propositions}
\begin{prop}
\label{prop:asymp exp Lambda}
For $\sigma>0$ sufficiently small:
\begin{enumerate}

\item For $1\le m\le n$,
%\fixme{In effect, this is the cubic expansion of $\Lambda_{m,n}(\cdot)$,
%with a bound on the $4$th derivative for the error term, but it seems extremely long to differentiate the system for the $4$th (!) time.
%It seems that there are two powers of $m$ drop for every {\em two} derivations of $\Lambda_{m,n}(\cdot)$, due to the even derivatives of $\frac{\tan{x}}{x}$. (unlike the square, where it is so for every derivation).}
\begin{equation}
\label{eq:3 term exp Lambda}
\Lambda_{m,n}(\sigma) = \Lambda_{m,n}(0) + \frac{4}{r}\cdot \sigma -
\frac{4 F_R(m,n)  }{\pi^{2} } \cdot \sigma^{2}(1-r\sigma) + O\left(\frac{1}{m^{4}}\cdot \sigma^{3}\right),
\end{equation}
with the implied constant absolute.

\item For $m=0<n$, $\Lambda_{0,n}(\cdot)$ satisfies\footnote{By general perturbation theory ~\cite[Chapter VII]{Kato}
(see in particular Remark 4.22 on page $408$), the functions $\Lambda_{0,n}(\cdot)$ are analytic, at least in some neighbourhood of the origin. Therefore the remainder term in \eqref{eq:Lambda0,n asympt} can be replaced by  $O_{n}(\sigma^{2})$.}
\begin{equation}
\label{eq:Lambda0,n asympt}
\Lambda_{0,n}(\sigma) = \Lambda_{0,n}(0)+ \frac{10}{3r}\cdot \sigma + O(\sigma^{3/2}),
\end{equation}
with the implied constant absolute.

\item The function $\Lambda_{0,0}(\sigma)$ is continuous\footnote{In fact, $\Lambda_{0,0}(\cdot)$ is analytic, by~\cite[Chapter VII]{Kato}.}
%\footnote{\textcolor{red}{this is a consequence of general theory}}
at $\sigma=0$.
\end{enumerate}
\end{prop}

\begin{prop}\label{prop:2nd term asymp}
If $(m,n)$ and $(m',n')$ are two integer points on the ellipse $$X^2+XY+Y^2=R^2$$ with $1\leq m<m'\leq n'<n$, then $F_R(m,n)>F_R(m',n')$ and as $R\to \infty$ we have a lower bound for the difference
\[
F_R(m,n)-F_R(m',n')  \gg \frac 1{m^4}
\]
with the implied constant absolute.
\end{prop}

\subsection{Proof of Theorem \ref{thm:simplicity} assuming Propositions~\ref{prop:asymp exp Lambda}-\ref{prop:2nd term asymp}}
\begin{proof}

We assert that for any integers $m,n,m',n'\geq 0$, one has:

(i) If $$m^{2}+n^{2}+mn < m'^{2}+n'^{2}+m'n',$$ then $\Lambda_{m,n}(\sigma)<\Lambda_{m',n'}(\sigma)$ for $\sigma\in (0,\pi^2/(27 r))$.

(ii) For $\sigma>0$ sufficiently small (absolute), if  $0\le m<m'\le n'<n$ satisfy $$m^{2}+n^{2}+mn = m'^{2}+n'^{2}+m'n'$$
then $\Lambda_{m,n}(\sigma)>\Lambda_{m',n'}(\sigma)$.

The desymmetrized Neumann spectrum $\{\Lambda_{m,n}(0):0\leq m\leq n\}$ is partitioned into clusters of coinciding eigenvalues
\[
\mathcal C_R = \{\Lambda_{m,n}(0): 0\leq m\leq n, m^2+mn+n^2=R^2\} .
\]
Part (i) deals with the situation that at $\sigma=0$, we start from different clusters $\mathcal C_R$, $\mathcal C_{R'}$ with $R<R'$, and the claim is that there is some $\sigma_0$ so that for $\sigma\in (0,\sigma_0)$, the evolved clusters remain separate.
Since distinct integers are spaced at least one apart from each other, the distance between different Neumann clusters ($\sigma=0$)
is at least $4\pi^2/(27r^2)$. We use   our upper bound (Theorem~\ref{thm: diffs for triangle}) on the Robin-Neumann gaps:
 \[
\lambda_n(\sigma)-\lambda_n(0) < \frac{4}{r}\sigma,
\]
 so that if $4\sigma/r<4\pi^2/(27r^2)$ then different Robin clusters cannot mix, that is
\[
\Lambda_{m,n}(\sigma)<\Lambda_{m',n'}(\sigma)
\]
 for $\sigma\in (0,\pi^2/(27 r))$.
%\marginpar{note the new proof of part (i)}

Now take integers $0\le m<m'\le n'<n$, so that $$m^{2}+n^{2}+mn = m'^{2}+n'^{2}+m'n'$$
(equivalently, $\Lambda_{m,n}(0)=\Lambda_{m',n'}(0)$).
If $m=0$, then we invoke Proposition \ref{prop:asymp exp Lambda}(1)-(2) to write
\begin{equation*}
\Lambda_{m',n'}(\sigma)-\Lambda_{0,n}(\sigma) =  \frac{2}{3r}\cdot\sigma +O(\sigma^{3/2}) >0,
\end{equation*}
since $F(m',n')<3$.
%$$\frac{m'^{2}+n'^{2}+m'n'}{\pi^{2} m^{2}n^{2}(m+n)^{2}}$$ is bounded by an absolute constant (recall \eqref{eq:R radius def};
%in this context $R=R(m',n')$).

Otherwise, if $m\ge 1$, then we invoke Proposition \ref{prop:asymp exp Lambda}(1) to yield
\begin{equation*}
\Lambda_{m',n'}(\sigma)-\Lambda_{m,n}(\sigma) =\left(  F_R(m,n)-F_R(m',n')\right)\cdot  \frac{4\sigma^{2}(1-r\sigma) }{\pi^{2}}+ O\left( \frac{\sigma^{3}}{m^{4}}  \right)
\end{equation*}
which along with Proposition \ref{prop:2nd term asymp} show that for $\sigma>0$ \new{sufficiently small},
\[
\Lambda_{m',n'}(\sigma)-\Lambda_{m,n}(\sigma) \gg  \frac{\sigma^2}{m^4}+ O\left( \frac{\sigma^{3}}{m^{4}}  \right) \gg   \frac{\sigma^{2}}{m^{4}} >0
\]
in particular this difference is nonzero. In either case, $m=0$ or $m\ge 1$, (ii) is proved.
\end{proof}

\section{Asymptotic expansion of the eigenvalue curves}
\subsection{Some auxiliary results}
 We state some lemmas on the properties of the auxiliary parameters $M$, $N$  and $L$, which we then use to prove
 Proposition \ref{prop:asymp exp Lambda}.

\begin{lem}\label{lem:LMN->0}

For every $0\le m\le n$ and $\sigma\ge 0$ there exists a unique solution $(L,M,N)$ to \eqref{eq:sec eq} in the prescribed range \eqref{eq:LMN range}. These solutions satisfy:

\begin{enumerate}

\item
For $0 \le m< n$, $\sigma>0$ one has $L,N = O\left( \frac{\sigma}{n}\right)$, with the implied constant  absolute.

\item In addition to the above, $M=O\left( \frac{\sigma}{m}\right)$ uniformly for $1\le m\le n$, $\sigma>0$.
Otherwise, for $m=0<n$, one has $M= O\left(\sqrt{\sigma}\right)$ for $\sigma>0$ sufficiently small,
with the implied constant  absolute.

\item For $m=n=0$ one has $|L|,|M|,|N| \ll \sqrt{\sigma}$, so, in particular the functions $L,M,N$ are continuous at $\sigma = 0$.

\end{enumerate}

\end{lem}

Lemma \ref{lem:LMN->0} implies in particular, that, as $\sigma\rightarrow 0$, one has $$L(\sigma),M(\sigma),N(\sigma) \rightarrow 0$$ {\em uniformly} w.r.t. $m,n\ge 0$. Therefore
%$\lambda_{m,n}(\sigma)\rightarrow\lambda_{m,n}(0)=-(m+n),$
$\mu_{m,n}(\sigma) \rightarrow \mu_{m,n}(0)=m$ and $\nu_{m,n}(\sigma)\rightarrow \nu_{m,n}(0)=n$ uniformly.
It will also follow a  fortiori from our analysis below that uniformly in $m,n$,
\[
\lim_{\sigma \to 0} \Lambda_{m,n}(\sigma) =\Lambda_{m,n}(0)=\frac{4\pi^2}{27r^2} \left( m^{2}+n^{2}+mn \right) .
\]
% {\em uniformly}, but not from the above results, in case $n/m$ is large (or $m=0$ and $n\rightarrow \infty$).
\begin{lem}\label{lem:exp LMN m=0}
For $m=0$, $n\geq 1$ the functions $L,M,N$ are analytic on $\sigma>0$ and continuous at $\sigma=0$, with $L,N$ continuously differentiable on $\R_{\ge 0}$. Further, $L,M,N$ satisfy the following asymptotics around the origin, with all the implied constants  absolute:
\begin{equation*}%\label{eq:L 2term asymp}
N,-L=\frac{3r}{n\pi}\cdot \sigma
+O\left(\frac{\sigma^{3/2}}{n^{2}}\right), \quad M=\sqrt{3r/2} \cdot\sqrt{\sigma} +  O(\sigma^{3/2}).
\end{equation*}
%and
%\begin{equation*}%\label{eq:M 2term asymp}
%M=\sqrt{3r/2} \cdot\sqrt{\sigma} +  O(\sigma^{3/2}).
%\end{equation*}
%\item
%\begin{equation}
%\label{eq:N 2term asymp}
%N=\frac{3r}{ n\pi}\cdot \sigma +O\left(\frac{\sigma^{3/2}}{n^{2}}\right).
%\end{equation}
\end{lem}

To state the next lemmas, we use the uniform notation as in \S~\ref{sec: upper bound}, where we rewrite the  coupled system \eqref{eq:sec eq}
in compact form as follows: Set
%\marginpar{do we need $\lambda$?}
\[
\begin{split}
 m_1=m, \quad   & m_2=n, \quad  m_3= -(m+n),
 \\
\mu_1=\mu, \quad  &\mu_2=\nu,\quad  \mu_3 =-(\mu_1+\mu_2),
\\
 M_1=M, \quad &M_2=N, \quad M_3=L,
\end{split}
\]
so that
\[
\mu_1+\mu_2 +\mu_3=0=m_1+m_2+m_3
\]
and
\[
\mu_j = m_j + \frac 1\pi(2M_j-M_i-M_k),
\]
where $\{i,j,k\}=\{1,2,3\}$. Thus for each $0\leq m_1\leq m_2$, we obtain
   a coupled system for the variables $M_1,M_2,M_3$ with   $M_1,M_2,-M_3\in [0,\pi/2)$
 \begin{equation}\label{new coupled system}
    \mu_j \tan M_j =\frac{3r\sigma}{\pi}, \quad j=1,2,3  .
\end{equation}
%The first part of Lemma 1.5 in this  notation becomes
\begin{lem} \label{new lemma 1.5}

If $1\leq m_1\leq m_2$ then the   derivatives at $\sigma=0$ are
\begin{equation}\label{Mj'}
M_j'(0) = \frac{3r}{\pi m_j},
\end{equation}
and
\begin{equation}\label{Mj''}
 M_j''(0) = -\frac{18 r^2}{\pi^3} \frac{m_j^2+2m_i m_k}{m_j^3m_im_k}
\end{equation}
\end{lem}

We next give bounds for the  derivatives of $M_j$:
 \begin{lem} \label{the second part of Lemma 1.5. }
There is some $\sigma_0>0$ so that if $1\leq m_1\leq m_2$ then  $M_j(\sigma)$ are  analytic in $[0,\sigma_0]$ and satisfy
(uniformly in $[0,\sigma_0]$)
\begin{equation}\label{first deriv of M}
 M_j'=  \frac{3r  }{\pi m_j}\left( 1+O\left(\frac 1{m_j}\right) \right),
\end{equation}
\begin{equation}\label{second deriv of M}
 |M_j''|\ll \frac 1{m_1 m_j^2} ,
\end{equation}
\begin{equation}\label{M_j'''}
M_j''' = 2\left(M_j' \right)^3 + O\left(\frac 1{m_1^3m_j^2}\right)=  \frac{54 r^3}{\pi^3 m_j^3} + O\left( \frac 1{m_1^3m_j^2}\right) .
\end{equation}
\end{lem}

 %\marginpar{This is your Lemma 1.6}
\begin{lem}\label{lem:2term exp m,n>=1}
The values of the first three derivatives of $\Lambda_{m,n}(\cdot)$ at the origin are:

\begin{equation}\label{lem:2term exp m,n>=1 part 1}
\Lambda_{m,n}'(0) = \frac 4r
\end{equation}
\begin{equation}\label{lem:2term exp m,n>=1 part 2}
\Lambda_{m,n}''(0) = -\frac{8}{\pi^2}F_R(m,n)
%-\frac{8R^4}{\pi^2 m^2n^2(m+n)^2} =  -\frac{8R^4}{\pi^2 m_1^2m_2^2 m_3^2}
\end{equation}
\begin{equation}\label{lem:2term exp m,n>=1 part 3}
\Lambda_{m,n}^{(3)}(\sigma) = \frac{24 r}{\pi^2}F_R(m,n) +O(\frac 1{m^4}) .
%\frac{24 R^4r}{\pi^2 m_1^2m_2^2 m_3^2}  +O(\frac 1{m_1^4})
\end{equation}
\end{lem}

The proofs of these Lemmas will be given in \S~\ref{sec: pfs of lemmas 1 and 2}, \S \ref{sec: pfs of lemmas 3 and 4} and
\S \ref{sec: pf of lemma5}.

 \subsection{Proof of Proposition \ref{prop:asymp exp Lambda}}
\begin{proof}
% \fbox{\textcolor{red}{Redo this}}

Proposition \ref{prop:asymp exp Lambda}(1) is a direct consequence of \eqref{lem:2term exp m,n>=1 part 1} and \eqref{lem:2term exp m,n>=1 part 2}
via a three-term Taylor expansion around $\sigma=0$, invoking the Lagrange form of the remainder appealing to the estimate
\eqref{lem:2term exp m,n>=1 part 3}: For every $\sigma>0$ sufficiently small, one has the estimate
\begin{equation*}
\Lambda_{m,n}(\sigma) = \Lambda_{m,n}(0) + \frac{4}{r}\cdot \sigma -
\frac{4 F_R(m,n)}{\pi^{2} }\cdot \sigma^{2} + \left(\frac{4 r F_R(m,n)}{\pi^{2}}+ O\left(\frac{1}{m^{4}}\right)\right)\cdot \sigma^{3},
\end{equation*}
%\marginpar{shouldn't this be big O of $\sigma^3$?}
which yields \eqref{eq:3 term exp Lambda}.
Part (3) of Proposition \ref{prop:asymp exp Lambda} is a direct consequence of Lemma \ref{lem:LMN->0}(3).

 To prove Proposition \ref{prop:asymp exp Lambda}(2), namely that
 \[
 \Lambda_{0,n}(\sigma)= \Lambda_{0,n}(0) + \frac{10}{3r}\cdot\sigma + O(\sigma^{3/2})
 \]
 we write
 \[
 \Lambda_{0,n}(\sigma)= \frac{2\pi^2}{27r^2}\sum_{j=1}^3 \mu_j^2
 \]
 Using $m=0$, $n\geq 1$ we have
 \[
 \mu_1 = \frac 1\pi(2M_1-M_2-M_3)
 \]
 \[
  \mu_2 = n+\frac 1\pi(2M_2-M_3-M_1),\quad \mu_3 = -n +\frac 1\pi(2M_3-M_1-M_2)
 \]
 so that
 \[
 \sum_{j=1}^3 \mu_j^2= 2n^2 + \frac {6n}\pi ( M_2-M_3) +\frac 1{\pi^2} \sum_{j=1}^3 (2M_j-M_i-M_k)^2 .
 \]
 Inserting Lemma~\ref{lem:exp LMN m=0} which asserts that
 \[
M_1 = \sqrt{\frac{3r\sigma}{2}} +O(\sigma^{3/2}),
\quad M_2,-M_3 = \frac{3r}{n\pi}\sigma + O(\frac{\sigma^{3/2}}{n^2})
 \]
 gives
 \[
 \sum_{j=1}^3 \mu_j^2= 2n^2 +\frac{45r}{\pi^2}\sigma + O(\sigma^{3/2})
 \]
 so that
 \[
  \Lambda_{0,n}(\sigma)= \frac{2\pi^2}{27r^2}\sum_{j=1}^3 \mu_j^2 =\Lambda_{0,n}(0) + \frac{10r}{3}\sigma + O(\sigma^{3/2})
 \]
 as claimed.
 \end{proof}

\section{Proofs of lemmas \ref{lem:LMN->0}-\ref{lem:exp LMN m=0}}\label{sec: pfs of lemmas 1 and 2}

\subsection{Proof of Lemma~\ref{lem:LMN->0}}

\begin{proof}
The existence and uniqueness of the solutions $(L,M,N)$ to \eqref{eq:sec eq} was established in ~\cite[\S 6]{McCartin2004}.
We observe that $L\le 0$ and $M,N\ge 0$ forces $2L-M-N\le 0$, and so
$$2L-M-N - (m+n)\pi \le -(m+n)\pi.$$
Hence, the first equation of \eqref{eq:sec eq} implies that if $(m,n)\ne (0,0)$ (which, for $m\le n$ is equivalent to $n\ne 0$), then $$|\tan{L}|\ll \frac{\sigma}{m+n}\le \frac{\sigma}{n},$$ and so $$L\ll \frac{\sigma}{n}.$$ The proof of $N\ll \frac{\sigma}{n}$ is similar to the above, except that we focus on the $3$rd equation of \eqref{eq:sec eq} and notice that
%, under the said constraints for $L,M,N$, one has
\[
2N-L-M+n\pi\ge n\pi - M \ge (n-1/2)\pi\gg n.
\]
This concludes the proof of Lemma \ref{lem:LMN->0}(1), and the same argument yields the case $m>0$ of Lemma \ref{lem:LMN->0}(2), by exploiting the $2$nd equation of \eqref{eq:sec eq}.

If $m=0$ but $n>0$, then, the $2$nd equation of \eqref{eq:sec eq} reads
\begin{equation}
\label{eq:2nd eq m=0}
(2M-L-N)\tan{M}=3r\sigma.
\end{equation}
First, assume by contradiction that $M<10\cdot(|L|+N)$ (say). Then, assuming that $\sigma>0$ is sufficiently small, so that, with the help from the (readily established) part (1) of Lemma \ref{lem:LMN->0}, $|L|,|N|< 1/100$ (so that $0\le M<1/5$), the l.h.s. of \eqref{eq:2nd eq m=0} is bounded above by
\begin{equation*}
(2M-L-N)\tan{M} \ll \frac{\sigma}{n} \cdot M \ll \frac{\sigma^{2}}{n^{2}},
\end{equation*}
so the equality \eqref{eq:2nd eq m=0} cannot hold with $\sigma>0$ \new{sufficiently} small. Hence we may as well assume that
$$M\ge 10\cdot (|L|+N).$$ But then we may deduce from \eqref{eq:2nd eq m=0}:
\begin{multline*}
M^{2} \ll (2M-M/10)\cdot M \le (2M-N)\cdot M
\\
 \ll (2M-N)\tan{M} \le (2M-L-N)\tan{M} = 3r\sigma,
\end{multline*}
so that $M\ll \sqrt{\sigma}$ as in the $2$nd assertion of Lemma \ref{lem:LMN->0}(2).

Finally we show Lemma \ref{lem:LMN->0}(3):  Since $m=n=0$, then the equality $M=N$ is forced
by the symmetry between these two. (If, by contradiction, $M>N$, then the l.h.s. of the $2$nd equation of \eqref{eq:sec eq}
is strictly bigger than the l.h.s. of the $3$rd equation of \eqref{eq:sec eq}.) Then the system \eqref{eq:sec eq}
reads
\begin{equation}
\label{eq:sec eq m=n=0}
\begin{cases}
2(L-M)\tan{L} &=3r\sigma \\
(M-L)\tan{M} &=3r\sigma
\end{cases}
\end{equation}
Then, by the $1$st equation of \eqref{eq:sec eq m=n=0},
and recalling that $L,\tan{L}\le 0$ and $M\ge 0$, it forces $2L\tan{L}\le 3r\sigma$, and so $L\ll \sqrt{\sigma}$, as above.
Further, either $M=L\ll \sqrt{\sigma}$ or we may divide the equations in \eqref{eq:sec eq m=n=0}, and  so
\[
M\ll\tan{M}=-2\tan{L}\ll \sqrt{\sigma},
\]
as we have already seen. This concludes the proof of Lemma~\ref{lem:LMN->0}(3).
\end{proof}

\subsection{ Proof of Lemma~\ref{lem:exp LMN m=0} }
%\fbox{\textcolor{red}{redo}}

Recall \eqref{eq:sec eq} (with $m=0$), namely
\begin{equation}\label{system with m=0}
\begin{split}
(2M-N-L)\tan M-3r\sigma &= 0 \\
(\pi n + 2N-L-M)\tan N-3r\sigma&=0 \\
(-\pi n + 2L-M-N)\tan L-3r\sigma &=0
\end{split}
\end{equation}
Consider the third equation of \eqref{system with m=0}:
 Thanks to Lemma \ref{lem:LMN->0}(1)
we may write $$\tan{L}=L+O(L^{3}) = L+O(\sigma^{3}/n^{3}),$$ and, in addition, from Lemma \ref{lem:LMN->0}(2), one has
\begin{equation*}
n\pi L = -3r\sigma +O(\sigma^{3/2}/n),
\end{equation*}
and then
\begin{equation}
\label{eq:L asympt 1term}
L = -\frac{3r}{n\pi} \cdot \sigma + O(\sigma^{3/2}/n^{2}).
\end{equation}
%that is identified as \eqref{eq:L 2term asymp}.
Analogously,
\begin{equation}
\label{eq:N asympt 1term}
N = \frac{3r}{n\pi} \cdot \sigma + O(\sigma^{3/2}/n^{2}) .
\end{equation}
%that is \eqref{eq:N 2term asymp}.

Next we focus on the first equation of \eqref{system with m=0}: We write
$$\tan(M) = M+\frac{1}{3}M^{3}+O(M^{5}) = M+\frac{1}{3}M^{3}+O(\sigma^{5/2}),$$ and feed \eqref{eq:L asympt 1term} and
\eqref{eq:N asympt 1term} into it to derive:
\begin{equation*}
%\label{eq:2M^2 asymp intermediate}
\begin{split}
2M^{2}&=3r\sigma +M(L+N)-\frac{2}{3}M^{4}-\frac{1}{3}M^{3}(L+N)+O(\sigma^{3})= 3r\sigma +O(\sigma^{2}) \\&=
3r\sigma(1+O(\sigma)),
\end{split}
\end{equation*}
so that
\begin{equation*}
M = \sqrt{\frac{3r}{2}}\cdot \sqrt{\sigma}(1+O(\sigma)) = \sqrt{\frac{3r}{2}}\cdot\sqrt{\sigma}+O(\sigma^{3/2}) .
\end{equation*}
%that is \eqref{eq:M 2term asymp}.

The continuity of $L,M,N$ at $\sigma=0$ follows directly from Lemma \ref{lem:LMN->0}, and here we deal with the analyticity of $L,M,N$ for $\sigma>0$ sufficiently small. We want to use the analytic Implicit Function theorem for the system \eqref{system with m=0}.
To do that we evaluate the Jacobian of \eqref{system with m=0} (with $m=0$) as
\[
J_{0,n}(\sigma)= \begin{pmatrix}
\frac{2M-N-L}{\cos^2 M}  +2\tan M & -\tan M & -\tan M \\
-\tan N &  \frac{\pi n + 2N-L-M}{\cos^2 N} & -\tan N \\
-\tan L & -\tan L & \frac{ -\pi n + 2L-M-N}{\cos^2 L} +2\tan L
\end{pmatrix} .
\]
For $\sigma \to 0$, we have $M,N,L \to 0$ so
\[
\tan M\sim M = O\left(\sqrt{\sigma}\right), \quad \tan N\sim N=O\left(\frac{\sigma}{n}\right),
\]
and likewise $\tan L =O(\frac{\sigma}{n}) $. Also
\[
\frac{1}{\cos^2M} = 1+\tan^2 M = 1+O(\sigma), \quad \frac{1}{\cos^2N}=1+O\left(\frac{\sigma^2}{n^2} \right),
\]
and likewise $1/\cos^2 L = 1+O(\sigma^2/n^2)$. Thus for small $\sigma>0$,
\[
J_{0,n}(\sigma)= \begin{pmatrix} 4M + O( \sigma ) & O(\sqrt{\sigma}) & O(\sqrt{\sigma}) \\
O(\frac \sigma n)&  \pi n + O(\sqrt{\sigma})  & O(\frac \sigma n)\\
O\left(\frac \sigma n \right)& O\left(\frac \sigma n \right) & -\pi n + O(\sqrt{\sigma})
\end{pmatrix}
\]
when $\sigma\to 0$, we obtain
\[
|\det J_{0,n}(\sigma)| = 4\pi^2 n^2 \cdot M  +O(\sigma)  \gg  \sqrt{\sigma}
\]
because $M\approx \sqrt{\sigma}$.
Thus we found $|\det J_{0,n}(\sigma)| \gg \sqrt{\sigma}\neq 0$ so that by the analytic Implicit Function Theorem, $M,N,L$ are analytic in $\sigma$ near $\sigma=0$.

\section{Proofs of Lemma~\ref{new lemma 1.5} and  Lemma~\ref{the second part of Lemma 1.5. }}\label{sec: pfs of lemmas 3 and 4}

\subsection{The derivatives of $M_j$ at $\sigma=0$: Proof of Lemma~\ref{new lemma 1.5}}

\begin{proof}
We compute derivatives: From the definition of $\mu_j$ we obtain
\[
\mu_j' = \frac 1\pi \left( 2M_j'-M_i'-M_k' \right) .
\]
From \eqref{new coupled system} we obtain after one differentiation
 \begin{equation}\label{first derivative M_j}
\mu_j' \tan M_j +\mu_j (\tan M_j)' = \frac{3r}{\pi}
 \end{equation}
 and differentiating again
 \begin{equation}\label{second derivative M_j}
 \mu_j'' \tan M_j +2\mu_j' (\tan M_j)' + \mu_j (\tan M_j)'' = 0 .
  \end{equation}
We also recall that
 \begin{equation}\label{values M_j  at sigma=0}
\mu_j(0) = m_j, \quad M_j(0) = 0
 \end{equation}
so that $\tan M_j(0)=0$, $\cos M_j(0)=1$.  Now $(\tan M_j)' = M_j'/(\cos^2 M_j)$ and therefore
 \[
 (\tan M_j)' (0) = M_j'(0) .
 \]

 Substituting in \eqref{first derivative M_j} and evaluating at $\sigma=0$ using  \eqref{values M_j  at sigma=0} gives
 $ m_j M_j'(0) =  \frac{3r}{\pi}$,  or
 \[
 M_j'(0) =  \frac{3r}{\pi m_j} ,
 \]
 which is \eqref{Mj'}.
 We also obtain
 \[
  \begin{split}
 \mu_j'(0) &=  \frac 1\pi \left( 2M_j(0)'-M_i'(0)-M_k'(0) \right)= \frac{3r}{\pi^2}\left(\frac{2}{m_j}-\frac 1{m_i}-\frac 1{m_k} \right)
\\
& =
  \frac{3r}{\pi^2}  \frac{2m_i m_k-m_j m_i- m_j m_k}{m_j m_i m_k}
  =
   \frac{3r}{\pi^2}  \frac{m_j^2+2m_i m_k }{m_j m_i m_k}
\end{split}
 \]
 on using $m_i+m_k=-m_j$.

 The second derivative of $\tan M_j$ is
 \begin{equation}\label{2nd derivative of tan M}
 (\tan M_j)'' = \left(\frac{ M_j'}{\cos^2 M_j}\right)' = \frac{M_j''}{\cos^2 M_j} -\frac{2 (M_j')^2 \tan M_j}{\cos^2 M_j}
 \end{equation}
 and at $\sigma=0$ we obtain
 \[
  (\tan M_j)''(0) = M_j''(0) .
 \]
  Inserting in \eqref{second derivative M_j} yields
 \[
 2\mu_j'(0) M_j'(0) + m_j M_j''(0) = 0
 \]
 or
 \[
m_j  M_j''(0) = -2 \frac{3r}{\pi m_j} \cdot \frac{3r}{\pi^2}  \frac{m_j^2+2m_i m_k }{m_j m_i m_k}
= -\frac{18r^2}{\pi^3}  \frac{m_j^2+2m_i m_k}{ m_j^2 m_i m_k}
 \]
 which gives
 \[
 M_j''(0) = -\frac{18r^2}{\pi^3}  \frac{m_j^2+2m_i m_k}{ m_j^3 m_i m_k}
 \]
 as claimed in \eqref{Mj''}.
 \end{proof}

 \subsection{Bounding derivatives of $M_j$: Proof of Lemma~\ref{the second part of Lemma 1.5. }}
 %Now for \textcolor{red}{the second part of Lemma 1.5. } \marginpar{change to different lemma}

%\marginpar{\textcolor{red}{add proof of real analyticity of $M_j$}}

Analyticity of $M_j$ near $\sigma=0$ is proved analogously to the case $m=0$, $n\geq 1$ in Lemma~\ref{lem:exp LMN m=0}.

We will prove the bounds on the derivatives.
Before proceeding, we formulate a standard fact from linear algebra (we leave the verification to the reader):
%\marginpar{Removed proof of Lemma~\ref{linear alg lem}}
\begin{lem}\label{linear alg lem}
 Suppose we have a system of the form $(I+B)  x =   y$, $  x,  y\in \R^n$  with $B$ a rank one matrix
 \[
 B=\beta \cdot \alpha^T, \quad \beta, \alpha \in \R^n
 \]
 where $||\alpha|| \cdot ||\beta||<1 $. % so that $I+B$ is invertible.
 Then
 \[
 x = y - \frac{\langle \alpha , y\rangle}{1+\langle \alpha, \beta \rangle}\beta
 \]
 and so
\[
 x_j  = y_j + O( ||y|| \cdot ||\alpha|| \cdot  |\beta_j|) .
\]
\end{lem}
\begin{comment}
\begin{proof}
We use the formula for the inverse \marginpar{\textcolor{red}{justify convergence }}
\[
(I+B)^{-1} = \sum_{s=0}^\infty (-B)^s
\]
Now to compute $B^s y$:
\[
By = \beta \alpha^T y = \langle \alpha,y \rangle \beta, \quad B^2 y = B(\langle \alpha,y \rangle \beta) =\langle \alpha,y \rangle \langle \alpha, \beta \rangle \beta
\]
and for any $s\geq 1$
\[
B^s y = \langle \alpha,y \rangle (\langle \alpha, \beta \rangle)^{s-1} \beta
\]
Therefore
\[
\begin{split}
x&=(I+B)^{-1}y = y+\sum_{s\geq 1}(-1)^s B^s y
\\
&= y + \sum_{s\geq 1} (-1)^s  \langle \alpha,y \rangle (\langle \alpha, \beta \rangle)^{s-1} \beta
\\
& =y- \frac{\langle \alpha,y\rangle}{1+ \langle \alpha, \beta \rangle} \beta
\end{split}
\]
\end{proof}
\end{comment}

  We now proceed with the proof of Lemma~\ref{the second part of Lemma 1.5. }:

\begin{proof}

 For the first derivative we  use   \eqref{first derivative M_j} and rearrange it as %(use $|\mu_j| \gg m_1$)
 \[
 M_j' \left( 1+\frac {6r \sigma \cos^2 M_j}{\pi^2}\frac 1{\mu_j^2} \right) - \frac{3r \sigma \cos^2 M_j}{\pi^2 \mu_j^2}(M_i'+M_k')  =
 \frac{3r \cos^2 M_j}{\pi \mu_j}
 \]
   that is, the vector $\vec M = (M_1,M_2,M_3)^T$ satisfies a matrix equation of the form
 \[
 (I +B)\vec M' = y
 \]
 with
 \[
 B = \frac{3r\sigma}{\pi^2} \begin{pmatrix} \frac {2 \cos^2 M_1}{\mu_1^2} & - \frac {\cos^2 M_1}{\mu_1^2}&  - \frac {\cos^2 M_1}{\mu_1^2} \\
  - \frac {\cos^2 M_2}{\mu_2^2} &   \frac {2\cos^2 M_2}{\mu_2^2}& - \frac {\cos^2 M_2}{\mu_2^2} \\
  - \frac {\cos^2 M_3}{\mu_3^2} & - \frac {\cos^2 M_3}{\mu_3^2}&  \frac {2\cos^2 M_3}{\mu_3^2}
 \end{pmatrix}
 = \beta \cdot \alpha^T
 \]
 where
 \[
 \alpha^T = \frac{3r\sigma}{\pi^2}\left(2,-1,-1 \right), \quad \beta^T =  \left( \frac{\cos^2 M_1}{\mu_1^2}, \frac{\cos^2 M_2}{\mu_2^2},  \frac{\cos^2 M_3}{\mu_3^2} \right),
 \]
 and
 \[
 y_j =   \frac{3r \cos^2 M_j}{\pi \mu_j} .
 \]
  We use Lemma~\ref{linear alg lem}, noting that
  \[
  |\langle \alpha,y \rangle |\leq |\alpha | \cdot |y| \ll \frac 1{m_1}, \quad |\beta_j|\ll \frac 1{m_j^2}
  \]
  to find
  \[
  M'_j =  \frac{3r \cos^2 M_j}{\pi \mu_j} +O\left(  \frac 1{m_1 m_j^2} \right)  =  \frac{3r  }{\pi m_j}\left( 1+O\left(\frac 1{m_j} \right) \right)
  \]
   locally uniformly in $\sigma$.

 For the second derivative, use
 \begin{equation}\label{rewrite 2n derivative of tan M}
 \cos^2 M_j (\tan M_j)'' = M_j'' -2(M_j')^2 \frac{3r\sigma}{\pi \mu_j}
 \end{equation}
 (which is a rewriting of \eqref{2nd derivative of tan M} using \eqref{new coupled system})
  and  insert into \eqref{second derivative M_j} (again using \eqref{new coupled system}) to obtain (recalling $(\tan M_j)'\cos^2 M_j = M_j'$)
\[
\mu_j \left( M_j'' -\frac{6r\sigma}{\pi} \frac{(M_j')^2}{\mu_j}\right) = -2 \mu_j'M_j' - \mu_j'' \cos^2 M_j \tan M_j
\]
or, after using  \eqref{new coupled system}
\begin{equation}\label{eq for M''}
M_j'' +\frac {3r\sigma\cos^2 M_j }{ \pi^2  \mu_j^2} \left(2M_j''-M_i''-M_k'' \right)  = -\frac {2 \mu_j'M_j'}{\mu_j}
+ \frac{6r\sigma}{\pi} \frac{(M_j')^2}{\mu_j}  .
\end{equation}
The RHS of \eqref{eq for M''} is $O(1/(m_1 m_j^2))$ by our bounds on the first derivative
while the LHS of \eqref{eq for M''} is of the form    $(I+B) \vec M''$ with $B=\beta \alpha^T$,
\[
\alpha^T = \frac{3r\sigma}{\pi^2}(2,-1,-1), \quad
\beta^T =   \left( \frac{\cos^2 M_1  }{\mu_1^2 }, \frac{\cos^2 M_2  }{\mu_2^2}, \frac{\cos^2 M_3  }{\mu_3^2}\right) .
\]
 Applying  Lemma~\ref{linear alg lem} we find
 \[
 M_j'' = -\frac {2 \mu_j'M_j'}{\mu_j} + \frac{6r\sigma}{\pi} \frac{(M_j')^2}{\mu_j}
 + O\left(\frac 1{m_1^3}\cdot \frac 1{ m_j^2} \right) \ll \frac 1{m_1m_j^2}  .
 \]

  \end{proof}

%\subsubsection{Third derivative}

For the third derivative we have:
\begin{lem}
\begin{equation*}%\label{M_j'''}
M_j''' = 2(M_j')^3 + O\left(\frac 1{m_1^3m_j^2} \right)=  \frac{54 r^3}{\pi^3 m_j^3} + O\left( \frac 1{m_1^3m_j^2} \right) .
\end{equation*}
\end{lem}
\begin{proof}
Differentiate \eqref{second derivative M_j} to obtain
\begin{equation}\label{third derivative M_j}
\mu_j (\tan M_j)'''  + 3\mu_j'(\tan M_j)'' +3\mu_j'' (\tan M_j)' +\mu_j''' \tan M_j=0 .
\end{equation}
Recall \eqref{2nd derivative of tan M} which gives
\begin{equation}\label{2nd derivative of tan M times cos^2}
\cos^2 M_j (\tan M_j)'' = M_j'' -2(M_j')^2 \tan M_j
\end{equation}
which in particular we now know to be $O(1/m_1 m_j^2)$ using \eqref{new coupled system}.
Differentiating \eqref{2nd derivative of tan M times cos^2} gives
\begin{multline*}
-2\tan M_j(\cos^2 M_j) M_j' (\tan M_j)'' + \cos^2 M_j (\tan M_j)'''  \\
= M_j''' -4M_j' M_j'' \tan M_j
-  2 (M_j')^2 (\tan M_j)'
\end{multline*}
so that by \eqref{second deriv of M}
\begin{equation}\label{new expression for M_j'''}
\cos^2 M_j (\tan M_j)'''  = M_j'''  -2(M_j')^3
+O\left( \frac 1{m_1 m_j^4}  \right)
\end{equation}
on using \eqref{new coupled system} and
\[
(\tan M_j)' = \frac{M_j'}{\cos^2 M_j} =M_j' \left(1+\left(\tan M_j \right)^2 \right) = M_j'  +O\left(\frac 1{m_j^3} \right) .
\]

Multiplying \eqref{third derivative M_j} by $(\cos^2 M_j)/\mu_j$ and inserting \eqref{new expression for M_j'''} gives
\[
 M_j'''
 + \mu_j''' \tan M_j \frac{\cos^2 M_j}{\mu_j}
= 2(M_j')^3 +O\left(\frac 1{m_1m_j^4} \right)
 \]
and thus we get
\[
M_j''' + O\left(\frac 1{m_j^2}\right) \left(2M_j'''-M_i'''-M_k''' \right) =   2(M_j')^3 +O\left(\frac 1{m_1m_j^4} \right) .
\]
Applying Lemma~\ref{linear alg lem} with $\alpha^T = (2,-1,-1)$, $\beta_j = O(1/m_j^2)$  and $y_j = 2(M_j')^3 + O(\frac 1{m_1m_j^4})$  (so that $\langle \alpha ,y \rangle =  O(1/m_1^3)$) gives
\[
M_j''' = 2(M_j')^3 + O\left(\frac 1{m_1^3m_j^2} \right) .
\]
 %\marginpar{for $j\neq 1$ the remainder may dominate the main term !}
  We then use \eqref{first deriv of M}
 \[
 M_j'  = \frac{3r}{\pi m_j} \left( 1+ O\left(\frac \sigma{m_j} \right) \right)
 \]
  to obtain
 \[
 M_j''' = \frac{54 r^3}{\pi^3 m_j^3} + O\left(  \frac 1{m_1^3m_j^2} \right) .
 \]
 \end{proof}

 \section{Proof of Lemma~\ref{lem:2term exp m,n>=1} }\label{sec: pf of lemma5}
Recall that
\[
\Lambda_{m,n} = \frac{2\pi^2}{27 r^2}\sum_{j=1}^3 \mu_j^2
\]
and we set
\[
R^2 = m^2+mn+n^2   = \frac 12\sum_{j=1}^3 m_j^2 ,
\]
\[
F_R(m,n) = \frac{R^4}{m^2n^2(m+n)^2} = \frac 1{m^2}+\frac 1{n^2}+\frac 1{(m+n)^2}.
\]
We want to show
\begin{lem} %\label{Lemma 1.6}
Assume that $1\leq m\leq n$. Then
\[
\Lambda_{m,n}'(0) = \frac 4r , \quad
\Lambda_{m,n}''(0) = - 8F_R(m,n) %=\frac{8R^4}{\pi^2 m^2n^2(m+n)^2} =  -\frac{8R^4}{\pi^2 m_1^2m_2^2 m_3^2}
\]
\[
\Lambda_{m,n}^{(3)}(\sigma) =24 \cdot r \cdot  F_R(m,n)+O \left(\frac 1{m^4} \right)
%  \frac{24 R^4r}{\pi^2 m_1^2m_2^2 m_3^2}  +O(\frac 1{m_1^4})
\]
\end{lem}
%\textcolor{red}{ A Mathematica computation reveals that
%\begin{equation*}%\label{mathematica 2}
%\sum_{j=1}^3 \frac 1{m_j^2} = \frac{R^4}{(m_1m_2m_3)^2}
%\end{equation*}
%Which version is better for us?}

\subsection{The first derivative of $\Lambda_{m,n}$}
Differentiating gives
 \[
 \Lambda_{m,n}'(0) =\frac{4\pi^2}{27 r^2}\sum_{j=1}\mu_j(0)\mu_j'(0) .
 \]
 We have
 \[
 \begin{split}
 \sum_{j=1}\mu_j(0)\mu_j'(0) &=\frac 1\pi \sum_{j=1}^3 m_j\left(2M_j'(0)-M_i'(0)-M_k'(0) \right)
 \\
 &= \frac{3r}{\pi^2} \sum_{j=1}^3 m_j  \left(\frac 2{m_j}-\frac 1{m_i}-\frac 1{m_k} \right)
 \end{split}
 \]
since $\mu_j(0) = m_j$, and $M_j'(0) = 3r/(\pi m_j)$. Since
\[
\sum_{j=1}^3 m_j  \left(\frac 2{m_j}-\frac 1{m_i}-\frac 1{m_k} \right)  = 6- \sum_{j=1}^3 m_j \left(\frac 1{m_i}+ \frac 1{m_k} \right) ,
\]
and
\begin{multline*}
- \sum_{j=1}^3 m_j \left(\frac 1{m_i}+ \frac 1{m_k} \right)  = -\sum_{j=1}^3 \frac{m_j}{m_i} -\sum_{j=1}^3 \frac{m_j}{m_k}
\\ =- \sum_{k=1}^3 \frac 1{m_k} \left( m_j+m_i \right) =\sum_{k=1}^3 1=3  ,
% \sum_{j=1}^3 \frac{m_j^2}{m_i m_k} = -\frac{m^2}{n(m+n)}-\frac{n^2}{m(m+n)}+\frac{(m+n)^2}{mn} = 3
\end{multline*}
we obtain
 \[
 \sum_{j=1}\mu_j(0)\mu_j'(0)=   \frac{27r}{\pi^2 }
 \]
 which gives
 \[
 \Lambda_{m,n}'(0)= \frac{4\pi^2}{27 r^2}\cdot  \frac{27r}{\pi^2}  = \frac{4}{r} .
 \]

 \subsection{The second derivative of $\Lambda_{m,n}$}
Using Leibnitz's rule gives
\[
\Lambda_{m,n}'' = \frac{4\pi^2}{27 r^2} \sum_{j=1}^3 \mu_j\mu_j''+(\mu_j')^2 .
\]
We have
\[
\mu_j(0) = m_j, \quad M_j'(0) =\frac{3r}{\pi m_j}, \quad
M_j''(0) %=  -\frac{18r^2}{\pi^3} \frac{m_j^2+2m_i m_k}{m_j^3m_im_k}
 =-\frac{18r^2}{\pi^3} \left(  \frac 1{m_im_jm_k}+\frac{2}{m_j^3}\right)
\]
so that
\[
\mu_j'(0) =\frac 1\pi(2M_j'(0)-M_i'(0)-M_k'(0)) = \frac{3r}{\pi^2} \left(\frac 2{m_j}-\frac 1{m_i}-\frac 1{m_k} \right)
\]
and
\begin{multline*}
 \mu_j(0)\mu_j''(0)  = m_j\frac 1\pi (2M_j''(0)-M_i''(0)-M_k''(0))
 \\
=  -\frac{18r^2}{\pi^4}
 \left( 2 \left(\frac {1}{m_i m_j}+\frac{2}{m_j^2} \right)
  -\left(\frac {1}{m_jm_k}+\frac{2m_j}{m_i^3} \right) -\left(\frac {1}{m_jm_i}+\frac{2m_j}{m_k^3}\right) \right)
\end{multline*}
giving
 \[
 \sum_{j=1}^3\mu_j(0)\mu_j''(0) =  -\frac{36 r^2}{\pi^4} \sum_{j=1}^3 \left( \frac{2}{m_j^2} -\frac{ m_j}{m_i^3}-\frac{ m_j}{m_k^3} \right)  .
 \]
Likewise,
\[
\sum_{j=1}^3 \mu_j'(0)^2 = \frac{9r^2}{\pi^4}\sum_{j=1}^3 (\frac 2{m_j}-\frac 1{m_i}-\frac 1{m_k})^2 .
\]

A straightforward computation reveals that
\begin{equation*}%\label{mathematica 1}
 \sum_{j=1}^3 \left( \frac{2}{m_j^2} -\frac{ m_j}{m_i^3}-\frac{ m_j}{m_k^3} \right)   = \frac{3 R^4}{(m_1m_2m_3)^2}=3F_R(m,n)
\end{equation*}
and
\[
\sum_{j=1}^3 \left(\frac 2{m_j}-\frac 1{m_i}-\frac 1{m_k}\right)^2=  \frac{6R^4}{(m_1m_2m_3)^2} =6F_R(m,n).
\]
These give
\[
\frac{4\pi^2}{27 r^2} \sum_{j=1}^3 \mu_j(0)\mu_j''(0) = - \frac{16}{\pi^2}F_R(m,n)
, \;
\frac{4\pi^2}{27 r^2}\sum_{j=1}^3 \mu_j'(0)^2  = \frac{8}{\pi^2}F_R(m,n) .
\]
Altogether we find
\[
\Lambda_{m,n}'' (0) =  - \frac{8}{\pi^2}F_R(m,n) .
\]

\subsection{The third derivative}
We have
\[
\Lambda_{m,n}^{(3)} = \frac{2\pi^2}{27 r^2} \sum_{j=1}^3 2\mu_j \mu_j'''   +6\mu_j'\mu_j''  .
\]
We use $\mu_j'\ll 1/m_1$, $\mu_j''\ll 1/m_1^3$ to deduce that
\[
\Lambda_{m,n}^{(3)} = \frac{4\pi^2}{27 r^2} \sum_{j=1}^3  \mu_j \mu_j'''  +O(\frac 1{m_1^4}) .
\]
Now
\[
 \sum_{j=1}^3  \mu_j \mu_j'''   = \frac 1\pi \sum_{j=1}^3 \mu_j(2M_j'''-M_i'''-M_k''')  =
 \frac 1\pi \sum_{j=1}^3 M_j'''(2\mu_j-\mu_i-\mu_k )
\]
after reordering the sum. We  use a simple lemma:
\begin{lem}\label{lem: summing 3 indices}
If $\{i,j,k\} = \{1,2,3\}$, and $b_1+b_2+b_3=0$, then
\[
\sum_{j=1}^3 a_j(2b_j-b_i-b_k) = 3\sum_{j=1}^3a_j b_j .
\]
\end{lem}
Apply Lemma~\ref{lem: summing 3 indices} to $a_j=M_j'''$, $b_j=\mu_j$, to obtain
\[
 \sum_{j=1}^3  \mu_j \mu_j'''   =
\frac 3\pi  \sum_{j=1}^3  \mu_j M_j'''
 \]
 Using $\mu_j = m_j +O( 1/m_j)$ and \eqref{M_j'''} which states that
\[
M_j''' = \frac{54 r^3}{\pi^3 m_j^3} +O\left(\frac 1{m_1^3m_j^2} \right)
\]
gives
\[
\mu_jM_j'''  = \frac{54 r^3}{\pi^3 m_j^2} +O\left(\frac 1{m_1^3m_j } \right)
=\frac{54 r^3}{\pi^3 m_j^2} +O\left(\frac 1{m_1^4 } \right) .
\]
We obtain
\[
 \sum_{j=1}^3  \mu_j \mu_j'''   =
\frac 3\pi  \sum_{j=1}^3 \frac{54 r^3}{\pi^3 m_j^2} +O\left(\frac 1{m_1^4 } \right)
\]
and so
\begin{multline*}
\Lambda_{m,n}^{(3)} = \frac{4\pi^2}{27 r^2} \sum_{j=1}^3  \mu_j \mu_j'''  +O\left(\frac 1{m_1^4} \right)
= \frac{4\pi^2}{27 r^2} \frac 3\pi  \sum_{j=1}^3 \frac{54 r^3}{\pi^3 m_j^2} +O\left(\frac 1{m_1^4 } \right)
\\
= \frac{24 r}{\pi^2} \sum_{j=1}^3 \frac 1{m_j^2} +O\left(\frac 1{m_1^4 } \right) = 24 rF_R(m,n) +O\left(\frac 1{m_1^4 } \right) .
\end{multline*}
This concludes the proof of Lemma~\ref{lem:2term exp m,n>=1} . \qed

 \section{Proof of Proposition~\ref{prop:2nd term asymp} }
%We set
%\[
%R^2:=m^2+mn+n^2
%\]
%and
%\[
%F_R(m,n) = \frac{R^4}{m^2n^2(m+n)^2}  = \frac 1{m^2}+\frac 1{n^2}+\frac 1{(m+n)^2}
%\]
%\begin{prop}\textcolor{red}{[Proposition 1.3]}
%If $(m,n)$ and $(m',n')$ are two integer points on the ellipse $X^2+XY+Y^2=R^2$ with $1\leq m<m'\leq n'<n$, then $F_R(m,n)>F_R(m',n')$ and as $R\to \infty$ we have a lower bound for the difference
%\[
%F_R(m,n)-F_R(m',n') >  \frac {59}{R^4}+O(\frac 1{R^5}) %\gg \frac 1{R^4}  .
%\]
%\end{prop}
\begin{proof}

We first assume that $$m'>10m.$$
%\boxed{\textcolor{red}{complete}}
Then use
\[
F_R(m,n) = \frac 1{m^2}+\frac 1{n^2}+\frac 1{(m+n)^2} >\frac 1{m^2},
\]
\[
 F_R(m',n') = \frac 1{m'^2}+\frac 1{n'^2}+\frac 1{(m'+n')^2}<\frac 3{m'^2} < \frac 3{(10 m)^2}
\]
to obtain
\[
F_R(m,n)-F_R(m',n')% > \frac 1{m^2}-\frac 3{m'^2}
>\frac 1{m^2}-\frac 3{(10 m)^2} \gg \frac 1{m^2}
\]
which is certainly sufficient.

%Next, assume $m<m'\leq 10m$ but $m\leq \delta R$, so that in particular, $m'<10\delta R$:  \fbox{\textcolor{red}{Fix this}}

Now assume that for $\delta>0$ very small (but fixed),
\[
  \delta R< m<m'\leq 10 m .
\]
We will show that
 \begin{multline*}
 F_R(m,n)-F_R(m',n') =  \frac 1{R^2 }\left(f\left(\frac mR\right)-f\left(\frac{m'}{R}\right)\right)
 \\
 >\frac{243}{4} \frac 1{R^4}
 +O\left(\frac 1{R^5}\right) \gg \frac 1{R^4}
 \end{multline*}
 and since $m>\delta R$, we obtain
 \[
 F_R(m,n)-F_R(m',n')\gg \frac{\delta^4}{m^4} \gg \frac 1{m^4}
 \]
as required.

Given $R$,  and $1\leq m< n$, with $m^2+mn+n^2=R^2$ we can express $n$ in terms of $m$ as
\[
n=\sqrt{R^2 - 3\left(\frac m2\right)^2}-\frac m2 .
\]
Therefore we can  write
\[
F_R(m,n) = \frac 1{R^2}f\left(\frac mR\right)
\]
where
%\marginpar{$f=1/g$ in Igor's notation}
\[
f(x) = \frac 1{x^2}+ \frac 1{ (\sqrt{1-3(\frac x2)^2}-\frac x2)^2} + \frac 1{ (\sqrt{1-3(\frac x2)^2}+\frac x2)^2}
%= \frac 1{x^2}+ \frac{2-x^2}{(1-x^2)^2} = \frac 1{x^2(1-x^2)^2}
\]
which simplifies to
\[
f(x) = \frac 1{x^2(1-x^2)^2} .
\]
%Indeed,
%\[
%f(x) = \frac 1{x^2}+ \frac 1{ (\sqrt{1-3(\frac x2)^2}-\frac x2)^2} + \frac 1{ (\sqrt{1-3(\frac x2)^2}+\frac x2)^2}
%= \frac 1{x^2}+ \frac{2-x^2}{(1-x^2)^2} = \frac 1{x^2(1-x^2)^2}
%\]
 The derivative of $f$ is
 \[
 f'(t) = -\frac{2(1-3t^2)}{t^3(1-t^2)^3}
 \]
 which is negative for $0<t<1/\sqrt{3}$, so that $f$ is decreasing in that range. Moreover, the second derivative is
 \[
 f''(t) = \frac{6 \left(7 t^4-4 t^2+1\right)}{t^4 \left(t^2-1\right)^4}=\frac{6 \left(3t^4+(1-2t^2)^2\right)}{t^4 \left(t^2-1\right)^4}>0
 \]
 which is positive, hence $f'(t)$ is increasing (and negative) and $-f'(t)$ is positive and decreasing.

 Let $x=m/R$, $x'=m'/R$.  Then
 \[
 F_R(m,n)-F_R(m',n') = \frac 1{R^2}\left(f(x)-f(x') \right)
 \]
 We separate two cases: $m'=n'$,  or $m'<n'$.

 If $m'=n'$ then $x'=1/\sqrt{3}$ (since then $R^2 = 3(m')^2$), and $f'(1/\sqrt{3})=0$. We expand $f(x)$ around $x'=1/\sqrt{3}$ to first order with  Lagrange remainder term
\[
f\left(x\right)-f\left(\frac 1{\sqrt{3}}\right) = f'\left(\frac 1{\sqrt{3}} \right)\left(x-\frac 1{\sqrt{3}} \right) +\frac{f''(t)}{2}\left(x-\frac 1{\sqrt{3}}\right)^2
\]
 for some $t\in (x,1/\sqrt{3})\subset[0,\frac 1{\sqrt{3}}]$. Hence using $\frac 1{\sqrt{3}}-x=\frac{m'-m}{R}\geq \frac 1R$ and a numerical finding that $ \min_{[0,1/\sqrt{3}]}f''(t)=119.167\dots $, we obtain
 \begin{equation}\label{case m'=n'}
 f(x)-f\left(\frac 1{\sqrt{3}} \right) \geq \frac 12 \left( \min_{[0,1/\sqrt{3}]}  f''(t)\right) \left(x-\frac 1{\sqrt{3}} \right)^2
 %\geq   \frac {  \min_{[0,1/\sqrt{3}]}   f''(t) }{2R^2}
  >\frac {59}{R^2} .
 \end{equation}

 If $m'<n'$  then we use  the mean value theorem, obtaining that for some $ x<t<x'$ we have
 \[
 F_R(m,n)-F_R(m',n') = \frac 1{R^2}\left(f\left(x\right)-f\left(x' \right) \right) %= \frac 1{R^2} (x-x')f'(t)
 =  \frac 1{R^2}(x'-x)\cdot (-f'(t))
 \]
We want to give lower bounds for $x'-x$ and for $-f'(t)$.

We have $x'-x=(m'-m)/R\geq 1/R$. Moreover, we claim that
\[
x'\leq \frac 1{\sqrt{3}}-\frac 1{2R}.
\]
Indeed, since $m'^2+m'n'+(n')^2=R^2$  with $n'>m'$ so that by integrality $n'\geq m'+1$, we have
\[
4R^2 = 3(m')^2+(m'+2n')^2\geq 3(m')^2+(m'+2(m'+1))^2 = 4 (1 + 3 m' + 3 (m')^2)
\]
so that $x'=m'/R $ satisfies $3(x')^2+\frac 3R x' +\frac 1{R^2}\leq 1$, giving
\[
%\frac 1R\leq
x'\leq \frac 16\left(\sqrt{12-\frac 3{R^2}}-\frac 3R\right)<  \frac 1{\sqrt{3}}-\frac 1{2R} .
\]

Hence
 \[
 -f'(t)>-f'(x')>-f'(\frac 1{\sqrt{3}}-\frac 1{2R})=\frac{243}{4R}+O(\frac 1{R^2}) .%\gg \frac 1R
 \]
 %and   since $x'-x=(m'-m)/R\geq 1/R$ we obtain
 %\[
 %f(x)-f(x') >(x'-x)\cdot (-f'(\frac 1{\sqrt{3}}-\frac 1{2R})) .
% \]
 %Now $x<1/\sqrt{3}-1/(2R)$ \textcolor{red}{ (explain why here?)} so that
 %\[
 %-f'(x) >-f'(\frac 1{\sqrt{3}}-\frac 1{2R}) =\frac{243}{4R}+O(\frac 1{R^2}) \gg \frac 1R
 %\]
 %and therefore since $x'-x=(m'-m)/R\geq 1/R$ we obtain
 %\fbox{Wrong - need $=f'(t)>-f'(x')$} and then we don't know that $m'\neq n'$ !
  Therefore
 \begin{equation}\label{case m'<n'}
 f(x)-f(x') =(x'-x) \cdot(-f'(t))\geq \frac{243}{4} \frac 1{R^2}+O(\frac 1{R^3}) %\gg \frac 1{R^2}
 \end{equation}
 in this case.

 Combining \eqref{case m'=n'} and \eqref{case m'<n'} gives that in both cases,
 \[
 F_R(m,n)-F_R(m',n') =  \frac 1{R^2 }\left(f\left(\frac mR\right)-f\left(\frac{m'}{R}\right)\right)>  \frac {59}{R^4}+O(\frac 1{R^5}) %\gg \frac 1{R^4}
 \]
 as claimed.

Finally assume
\[
m<m'\leq 10m \quad \mbox{and} \quad  m\leq \delta R,
\]
so that in particular, $m'<10\delta R$:
Then
  \[
 F_R(m,n)-F_R(m',n') = \frac 1{R^2}(f(x)-f(x')) = \frac 1{R^2} (x'-x) \cdot (-f'(t))
 \]
 for some $t\in (x,x')$. Note that  $0<t<x'<10x<10\delta$ so that
 \[
 -f'(t) =\frac{2(1-3t^2)}{t^3(1-t^2)^3}  >\frac{2(1-(10 \delta)^2)}{t^3}>\frac 1{t^3}
 \]
 for $\delta>0$ sufficiently small,  so since $tR<m'<10m$ we obtain
 \[
 F_R(m,n)-F_R(m',n') =\frac{m'-m}{R^3} (-f'(t))>\frac{m'-m}{t^3R^3}  > \frac 1{(10m)^3}
 \]
which is consistent with the assertion of Proposition \ref{prop:2nd term asymp} in this case.
\end{proof}

\end{document}